%% file: SMOOTH_HARD_THRESHOLDING_FOR_SVD_WITH_SURE.tex
\pdfoutput=1
\documentclass[12pt,reqno]{amsart}
\usepackage{amsmath,amsthm,amssymb,amsfonts,mathtools}
\usepackage{bm}
\usepackage{enumitem}
\usepackage{graphicx,latexsym}
\usepackage{algorithm,algorithmic}
\usepackage{booktabs}
\usepackage{array}
\usepackage{url}
\numberwithin{equation}{section}

\makeatletter
\def\subsection{\@startsection{subsection}{2}{\z@}%
  {.7\linespacing\@plus.3\linespacing}{.35\linespacing}%
  {\normalfont\bfseries\def\@secnumfont{\bfseries}}}
\def\paragraph{\@startsection{paragraph}{4}{\z@}%
  {.45\linespacing\@plus.2\linespacing}{.2\linespacing}%
  {\normalfont\bfseries}}
\makeatother

\makeatletter
\def\@captionfont{\scriptsize}
\makeatother

\theoremstyle{plain}
\newtheorem{theorem}{Theorem}[section]
\newtheorem{lemma}[theorem]{Lemma}
\newtheorem{proposition}[theorem]{Proposition}

\theoremstyle{definition}
\newtheorem{definition}[theorem]{Definition}
\newtheorem{remark}[theorem]{Remark}

\newcommand{\diagop}{\operatorname{diag}}
\newcommand{\divop}{\operatorname{div}}
\newcommand{\R}{\mathbb{R}}
\newcommand{\E}{\mathbb{E}}
\newcommand{\ind}{\mathbf{1}}
\newcommand{\norm}[1]{\left\lVert #1\right\rVert}
\newcommand{\inner}[2]{\left\langle #1,#2\right\rangle}

\title[Smooth Hard-Thresholding for Singular Values with SURE]{Smooth Hard-Thresholding for Singular Values with Stein\textquoteright{}s Unbiased Risk Estimate}
\author[Yang, Wang, and Yang]{Guanzhong Yang}
\address{Department of Mathematics\\ Imperial College London, UK}
\email{guanzhong.yang25@imperial.ac.uk}

\author{Zi Wang}
\address{Department of Bioengineering\\ Imperial-X, Imperial College London, UK}
\email{zi.wang@imperial.ac.uk}

\author{Guang Yang}
\address{Department of Bioengineering\\ Imperial-X, Imperial College London, UK}
\email{g.yang@imperial.ac.uk}
\thanks{Zi Wang and Guang Yang are co-corresponding authors.}
\subjclass[2020]{Primary 62H12; Secondary 62G08, 15A18, 68U10}
\keywords{Stein's unbiased risk estimate, singular value thresholding, spectral estimators, low-rank matrix denoising, smooth hard thresholding}

\begin{document}

\begin{abstract}
Low-rank matrix denoising is a central primitive in patch-based image restoration and many other inverse problems. Classical SVD-based image denoising methods often choose a truncation rank by matching residual singular-value energy with an estimated noise energy, but this rule is not a finite-sample risk principle because a fitted low-rank approximation inevitably absorbs part of the noise. This paper develops a mathematically rigorous alternative based on Stein's unbiased risk estimate (SURE). Since singular value hard thresholding is discontinuous and does not satisfy the hypotheses of Stein's lemma, we introduce a logistic smooth hard-threshold spectral estimator. We prove that the smooth shrinker satisfies the regularity conditions required by a spectral-estimator version of Stein's lemma, and therefore admits an exactly unbiased fixed-threshold risk estimate under Gaussian noise. For a fixed observed matrix and a finite set of candidate thresholds separated from the observed singular values, the ordering of the fixed-threshold smooth SURE objective eventually agrees with a simple limiting score. The limiting score has the same algebraic form as the biased hard-threshold SURE formula, but here it is used only as a computational device for ranking finite candidates. Selecting the minimizing threshold is a data-adaptive tuning step; the selected SURE value should not be interpreted as an unbiased risk estimate of the finally selected estimator.
\end{abstract}

\maketitle

\section{Introduction}

\paragraph{Notation.}
For matrices $A,B$ of the same size, we write
$\inner{A}{B}=\operatorname{tr}(A^TB)$ and
$\norm{A}_F=\inner{A}{A}^{1/2}$ for the Frobenius inner product and Frobenius norm.
All expectations are taken with respect to the Gaussian noise $W$, while the signal matrix $X$ is treated as deterministic.

\paragraph{Observation model.}
Let
\begin{equation}
    Y=X+W\in\R^{m\times n},
\end{equation}
where $X$ is an unknown deterministic matrix, usually assumed to be low-rank or approximately low-rank, and $W_{ij}\stackrel{\mathrm{i.i.d.}}{\sim}N(0,\tau^2)$. Write $k=\min(m,n)$ and use the thin singular value decomposition
\begin{equation}
\begin{gathered}
    Y=U\diagop(\sigma_1,\ldots,\sigma_k)V^{T}, \\
    U\in\R^{m\times k},\qquad V\in\R^{n\times k},
    \qquad U^TU=V^TV=I_k.
\end{gathered}
\end{equation}
with $\sigma_1\geq\cdots\geq\sigma_k\geq 0$. For continuous Gaussian noise, the singular values are distinct and positive with probability one whenever the usual full-rank condition is satisfied; throughout the main derivation we therefore work on the simple full-rank set
\begin{equation}
    \sigma_1>\sigma_2>\cdots>\sigma_k>0.
\end{equation}

\paragraph{SVD energy matching.}
SVD-based low-rank approximation has long been used in image denoising. Guo, Zhang, Zhang and Liu proposed an efficient image denoising method that groups similar patches by block matching, arranges each group as a matrix, and applies SVD truncation to exploit the resulting low-rank structure \cite{guo2016efficient}. The mathematical motivation is the Eckart--Young--Mirsky theorem: keeping the largest singular values gives the best low-rank approximation in Frobenius norm. In their denoising rule, the truncation rank is selected by comparing discarded singular-value energy with the expected noise energy, approximately $mn\tau^2$.

\paragraph{Finite-sample bias of the residual.}
This energy-matching principle is attractive but theoretically limited. If $\widehat X=\widehat X(Y)$ is fitted from the noisy observation, then
\begin{equation}
\begin{aligned}
    \E\norm{Y-\widehat X}_{F}^{2}
    &=\E\norm{X+W-\widehat X}_{F}^{2} \\
    &=\E\norm{\widehat X-X}_{F}^{2}+\E\norm{W}_{F}^{2}-2\E\inner{\widehat X-X}{W}.
\end{aligned}
\end{equation}
A denoising estimator that adapts to $Y$ generally fits some component of the noise, so the correlation term is not zero. Consequently, residual energy should not be treated as an exact finite-sample surrogate for the noise energy or for the mean squared error (MSE)
\begin{equation}
    R(\widehat X,X)=\E\norm{\widehat X(Y)-X}_{F}^{2}.
\end{equation}

\paragraph{SURE and the hard-thresholding obstruction.}
A principled way to estimate risk under Gaussian noise is Stein's unbiased risk estimate. Cand\`es, Sing-Long and Trzasko established SURE formulae for singular value thresholding and, more generally, spectral estimators satisfying mild differentiability and integrability assumptions \cite{candes2013unbiased}. The hard-threshold family is a natural target in the present setting because classical SVD denoising is truncation-based: the leading singular components are kept at their observed magnitudes, while smaller components are discarded. Soft thresholding has a well-developed SURE theory, but it shrinks every retained singular value and therefore represents a different spectral shrinkage family. If the aim is to stay close to rank-truncated SVD denoising while replacing residual-energy matching by a risk-based finite-candidate tuning principle, the natural computational goal is to compare candidate thresholds, equivalently possible retained ranks, by SURE-type scores and select the lowest-scoring candidate. The corresponding hard-threshold spectral rule is
\begin{equation}
    \widehat X_{\lambda}(Y)
    =U\diagop(f_{\lambda}(\sigma_1),\ldots,f_{\lambda}(\sigma_k))V^{T},
    \qquad
    f_{\lambda}(x)=x\ind\{x>\lambda\}.
\end{equation}
This rule, however, contains a jump discontinuity at $x=\lambda$. Hansen emphasized that differentiability almost everywhere is not enough for Stein's lemma; in particular, singular value hard thresholding does not satisfy the relevant condition, and the corresponding hard-threshold SURE expression is not an unbiased estimate of risk \cite{hansen2017comment}.

\paragraph{Fixed-rank comparison.}
For a deterministic rank $h$, reduced-rank estimators have their own degrees-of-freedom and SURE theory \cite{mukherjee2015degrees,hansen2017comment}, which could also provide a finite-rank ranking score. We instead follow the smooth-threshold route because the motivating denoiser is parametrized by a singular-value cutoff rather than directly by rank. Section 4 gives the detailed relation between the resulting limiting score and fixed-rank SURE.

\paragraph{Contribution and organization.}
The goal of this paper is therefore not to deny the biasedness of hard-threshold SURE. Instead, we introduce a smooth hard-threshold approximation, prove that it satisfies the hypotheses under which SURE is exactly unbiased for every fixed threshold, and then use a limiting finite-candidate argument to obtain a simple and computationally efficient threshold-selection rule. The resulting procedure is a SURE-based tuning strategy, not a claim that the minimized SURE value is an unbiased risk estimate after threshold selection. Section 2 defines the smooth spectral estimator and proves the regularity needed for Stein's identity. Section 3 derives the finite-candidate limiting rule. Section 4 clarifies the precise interpretation of the limiting score and the post-selection caveat. Section 5 reports numerical checks of the fixed-threshold identity and post-selection caveat, evaluates the oracle-style performance of the ranking rule, describes the complete SVD denoising pipelines, examines finite-$\omega$ smooth reconstruction and sensitivity to noise-level misspecification, gives a paired BSD68 statistical comparison of the two SVD rank-selection rules, and then gives an illustrative fixed-realization comparison. Section 6 concludes.

\section{Smooth Hard-Threshold Spectral Estimator}

\subsection{Estimator and candidate thresholds}
\paragraph{Estimator.}
For fixed threshold $\lambda\geq 0$ and smoothness parameter $\omega>0$, define
\begin{equation}
    f_{\omega,\lambda}(x)
    =\frac{x}{1+\exp[-\omega(x-\lambda)]},
    \qquad x\in\R.
\end{equation}
The associated spectral estimator is
\begin{equation}
    \widehat X_{\omega,\lambda}(Y)
    =U\diagop(f_{\omega,\lambda}(\sigma_1),\ldots,f_{\omega,\lambda}(\sigma_k))V^{T}.
\end{equation}
For every $x\neq\lambda$,
\begin{equation}
    f_{\omega,\lambda}(x)\longrightarrow f_{\lambda}(x)=x\ind\{x>\lambda\}
    \qquad\text{as }\omega\to\infty.
\end{equation}
Thus $f_{\omega,\lambda}$ is a smooth approximation to hard thresholding away from the discontinuity.

\begin{figure}[H]
\centering
\includegraphics[width=0.5\textwidth]{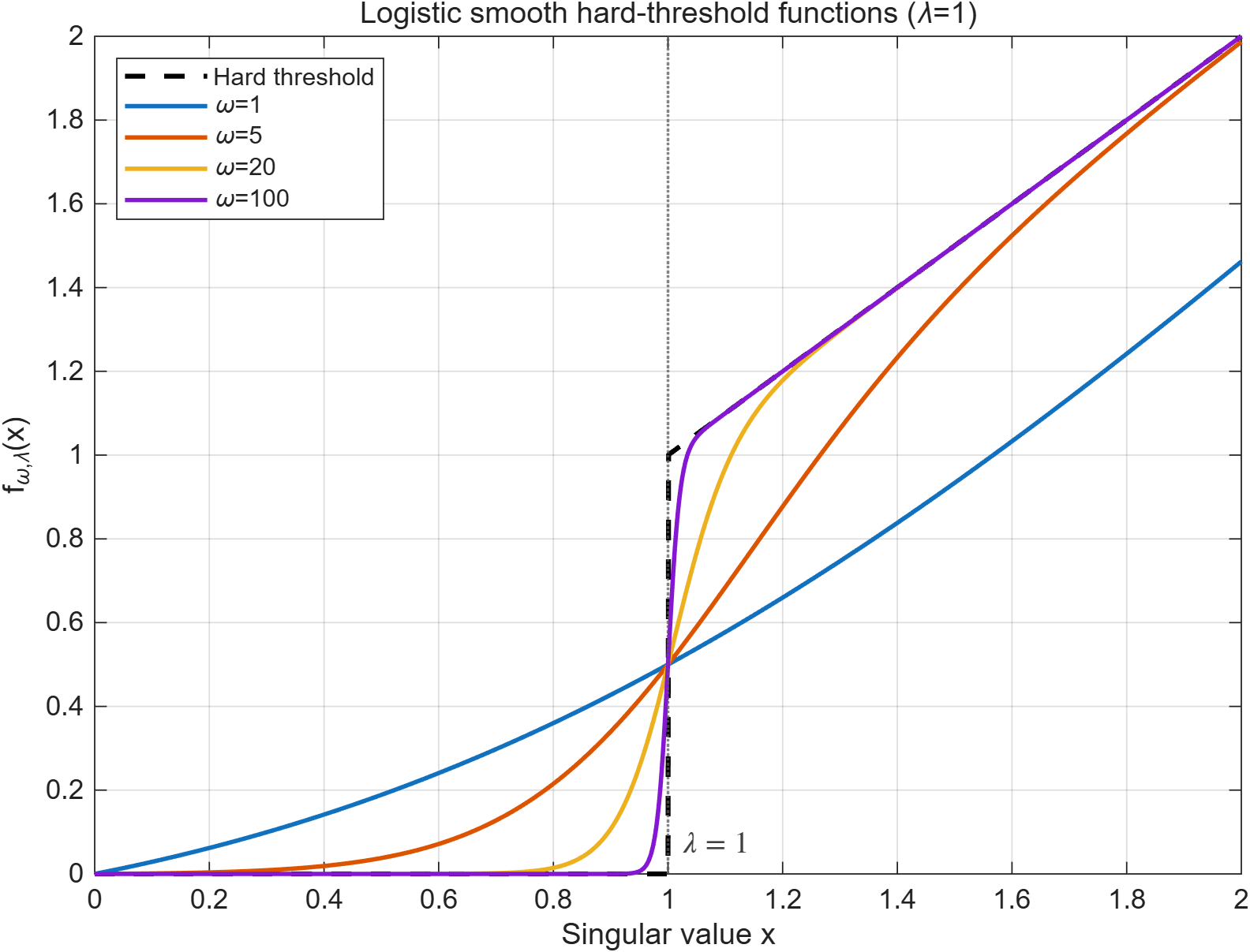}
\caption{Logistic smooth hard-threshold functions for $\lambda=1$. As $\omega$ increases, $f_{\omega,\lambda}$ approaches hard thresholding pointwise away from the discontinuity at $x=\lambda$.}
\label{fig:logistic-smooth-hard-threshold}
\end{figure}

\paragraph{Role of the logistic transition.}
The theory does not depend specifically on the logistic function. For example, let $q\in C^1(\R)$ satisfy
\[
    q(-\infty)=0,\qquad q(+\infty)=1,\qquad
    \sup_{t\in\R}|q'(t)|<\infty,
\]
and assume that $\sup_{t\in\R}|tq'(t)|<\infty$ and $tq'(t)\to0$ as $|t|\to\infty$. Then
\[
    \widetilde f_{\omega,\lambda}(x)=xq\bigl(\omega(x-\lambda)\bigr)
\]
is globally Lipschitz for every fixed $(\omega,\lambda)$, vanishes at the origin, and, for every $x\neq\lambda$, satisfies
\[
    \widetilde f_{\omega,\lambda}(x)\to x\ind\{x>\lambda\},
    \qquad
    \widetilde f'_{\omega,\lambda}(x)\to\ind\{x>\lambda\}.
\]
Under these conditions the regularity and finite-candidate limiting arguments below remain valid. The logistic transition is used because both the function and its derivative have simple closed forms and the required bounds and limits are immediate; an arbitrary smooth approximation need not satisfy these conditions.

\paragraph{Candidate thresholds.}
Given a simple full-rank observation $Y$, define the nonzero-rank threshold candidates by
\begin{equation}\label{eq:candidate-set}
    \mathcal B(Y)
    =\left\{\lambda_h=\frac{\sigma_h+\sigma_{h+1}}{2}:h=1,
    \ldots,k-1\right\}\cup\left\{\lambda_k=\frac{\sigma_k}{2}\right\}.
\end{equation}
For the zero-rank candidate, fix any realized numerical threshold $\lambda_0>\sigma_1$, equivalently the candidate that discards all singular components, and write
\[
    \mathcal B_0(Y)=\{\lambda_0\}\cup\mathcal B(Y).
\]
For $\lambda_h$ with $h\geq1$, exactly the first $h$ singular values are retained by the limiting hard-threshold rule.

\begin{remark}[Data-dependent candidate values]
The candidates in $\mathcal B_0(Y)$ are random before observation, whereas the SURE identity below applies to deterministic thresholds. Section 3 therefore treats the realized $\lambda_h(Y)$ only as plug-in numerical candidates; the statistical interpretation and post-selection caveat are given in Section 4.
\end{remark}

\paragraph{Why use a finite data-dependent candidate set?}
For finite $\omega$, continuously minimizing the smooth SURE criterion over $\lambda>0$ is a different tuning procedure: the smooth estimator varies with $\lambda$ even while $\lambda$ remains between the same two adjacent singular values, and the finite-candidate ordering result below does not characterize that continuous optimum. In the hard-threshold limit, however, every threshold in a fixed interval between adjacent observed singular values produces the same rank truncation and the same limiting score. Hence one representative from each interval, namely $\mathcal B_0(Y)$, exhausts all possible hard-truncation ranks. The final implementation may therefore minimize $S_h(Y)$ directly over $h=0,\ldots,k$; the smooth family supplies the fixed-threshold SURE justification for this ranking score, not an additional continuous optimization step. The post-selection interpretation is discussed in Section 4.

\subsection{Regularity and Unbiased SURE for Fixed Threshold}
\paragraph{Regularity of estimator.}
\begin{lemma}[Global Lipschitz continuity of the scalar shrinker]\label{lem:scalar-lip}
For every $\omega>0$ and $\lambda\geq 0$, the function $f_{\omega,\lambda}$ is globally Lipschitz on $\R$.
\end{lemma}

\begin{proof}
\emph{Derivative bound.} The derivative is
\begin{equation}
    f'_{\omega,\lambda}(x)
    =\frac{1}{1+e^{-\omega(x-\lambda)}}
    +\frac{\omega x e^{-\omega(x-\lambda)}}{(1+e^{-\omega(x-\lambda)})^2}.
\end{equation}
\emph{Compact-tail decomposition.} This derivative is continuous on $\R$. Moreover,
\begin{equation}
    \lim_{x\to\infty}f'_{\omega,\lambda}(x)=1,
    \qquad
    \lim_{x\to-\infty}f'_{\omega,\lambda}(x)=0.
\end{equation}
Hence there exists $M>0$ such that $|f'_{\omega,\lambda}(x)|\leq 2$ for $|x|>M$. On the compact interval $[-M,M]$, continuity gives $\sup_{|x|\leq M}|f'_{\omega,\lambda}(x)|<\infty$. Therefore $\sup_{x\in\R}|f'_{\omega,\lambda}(x)|<\infty$, and the mean value theorem implies global Lipschitz continuity.
\end{proof}

\begin{lemma}[Regularity of the spectral estimator]\label{lem:spectral-lip}
For every fixed $\omega>0$ and $\lambda\geq0$, the map $Y\mapsto\widehat X_{\omega,\lambda}(Y)$ is globally Lipschitz as a matrix-valued spectral function and satisfies $\widehat X_{\omega,\lambda}(0)=0$.
\end{lemma}

\begin{proof}
\emph{Origin condition.} The identity $\widehat X_{\omega,\lambda}(0)=0$ is immediate from $f_{\omega,\lambda}(0)=0/(1+e^{\omega\lambda})=0$. \emph{Spectral Lipschitz transfer.} By Lemma \ref{lem:scalar-lip}, the restriction of $f_{\omega,\lambda}$ to $[0,\infty)$ is Lipschitz and satisfies $f_{\omega,\lambda}(0)=0$. We use the operator-Lipschitz estimate for the singular value functional calculus of Andersson, Carlsson and Perfekt \cite{andersson2016operator}: if $f:[0,\infty)\to\R$ is Lipschitz and $f(0)=0$, then the singular-value map
\begin{equation}
    G_f(Y)=U\diagop(f(\sigma_1(Y)),\ldots,f(\sigma_k(Y)))V^T
\end{equation}
is Lipschitz with respect to the Frobenius norm. Hence $Y\mapsto\widehat X_{\omega,\lambda}(Y)$ is globally Lipschitz.
\end{proof}

\paragraph{Unbiased SURE identity.}
\begin{proposition}[Unbiased SURE for smooth hard thresholding]\label{prop:sure-unbiased}
Assume $Y=X+W$, where $\{W_{ab}:1\leq a\leq m,\ 1\leq b\leq n\}$ are independent and identically distributed as $N(0,\tau^2)$. For a weakly differentiable matrix-valued map $\widehat X:\R^{m\times n}\to\R^{m\times n}$, define its divergence by
\begin{equation}\label{eq:divergence-definition}
    \divop(\widehat X)(Y)
    =\sum_{a=1}^{m}\sum_{b=1}^{n}
        \frac{\partial \widehat X_{ab}(Y)}{\partial Y_{ab}},
\end{equation}
where the weak partial derivatives, and hence the displayed representative, are defined almost everywhere. For fixed $\omega>0$ and $\lambda\geq0$, the SURE functional
\begin{equation}\label{eq:sure-general}
    \operatorname{SURE}_{\omega,\lambda}(Y)
    =-mn\tau^2+\norm{\widehat X_{\omega,\lambda}(Y)-Y}_{F}^{2}
    +2\tau^2\divop(\widehat X_{\omega,\lambda})(Y)
\end{equation}
satisfies
\begin{equation}
    \E\operatorname{SURE}_{\omega,\lambda}(Y)
    =\E\norm{\widehat X_{\omega,\lambda}(Y)-X}_{F}^{2}.
\end{equation}
\end{proposition}

\begin{proof}
\emph{Application of Stein's identity.} Proposition III.1 of Cand\`es, Sing-Long and Trzasko states the Gaussian Stein/SURE identity under weak differentiability and integrability, while their Lemma III.2 shows that a Lipschitz map has the required weak regularity \cite{candes2013unbiased}. Hansen also stresses that global Lipschitz continuity is a sufficient condition for Stein's lemma in the matrix setting, after identifying $\R^{m\times n}$ with $\R^{mn}$ \cite{hansen2017comment}. Lemma \ref{lem:spectral-lip} gives this Lipschitz condition and the origin condition. To make the integrability explicit, let $L_{\omega,\lambda}<\infty$ be a Lipschitz constant for $Y\mapsto\widehat X_{\omega,\lambda}(Y)$. Since $\widehat X_{\omega,\lambda}(0)=0$,
\begin{equation}
    \norm{\widehat X_{\omega,\lambda}(Y)}_F
    \leq L_{\omega,\lambda}\norm{Y}_F,
\end{equation}
and therefore $\E\norm{\widehat X_{\omega,\lambda}(Y)-Y}_F^2<\infty$ under Gaussian noise. By Rademacher's theorem the Lipschitz map is differentiable almost everywhere and weakly differentiable; its weak Jacobian is essentially bounded by $L_{\omega,\lambda}$, so
\begin{equation}
    |\divop(\widehat X_{\omega,\lambda})(Y)|
    \leq mn L_{\omega,\lambda}
    \quad\text{for almost every }Y.
\end{equation}
Thus $\E|\divop(\widehat X_{\omega,\lambda})(Y)|<\infty$. The hypotheses of the Gaussian Stein identity are therefore satisfied, and \eqref{eq:sure-general} is an unbiased estimate of the fixed-threshold risk.
\end{proof}

\paragraph{Closed-form divergence.}
For a simple full-rank matrix, the divergence formula for the real spectral estimator
\[
    G_f(Y)=U\diagop(f(\sigma_1),\ldots,f(\sigma_k))V^T
\]
is
\begin{equation}\label{eq:divergence}
    \divop(G_f)(Y)
    =|m-n|\sum_{i=1}^{k}\frac{f(\sigma_i)}{\sigma_i}
    +\sum_{i=1}^{k}f'(\sigma_i)
    +2\sum_{\substack{i,j=1\\i\neq j}}^{k}
        \frac{\sigma_i f(\sigma_i)}{\sigma_i^2-\sigma_j^2},
\end{equation}
which is the real-valued divergence formula for spectral estimators given in Theorem IV.3 of \cite{candes2013unbiased}. Substituting $f=f_{\omega,\lambda}$ gives
\begin{equation}\label{eq:smooth-sure}
\begin{aligned}
    \operatorname{SURE}_{\omega,\lambda}(Y)
    ={}&-mn\tau^2+
    \sum_{i=1}^{k}\bigl(\sigma_i-f_{\omega,\lambda}(\sigma_i)\bigr)^2 \\
    &+2\tau^2\Biggl\{
    |m-n|\sum_{i=1}^{k}\frac{f_{\omega,\lambda}(\sigma_i)}{\sigma_i}
    +\sum_{i=1}^{k}f'_{\omega,\lambda}(\sigma_i) \\
    &\hspace{3.5cm}
    +2\sum_{\substack{i,j=1\\i\neq j}}^{k}
        \frac{\sigma_i f_{\omega,\lambda}(\sigma_i)}{\sigma_i^2-\sigma_j^2}
    \Biggr\}.
\end{aligned}
\end{equation}

\section{Threshold Selection over a Finite Candidate Set}

\subsection{Limiting scores}
For a fixed observed matrix $Y$, we evaluate the smooth fixed-threshold SURE formula at the realized candidates $\lambda_h(Y)$ and study their pointwise ordering as $\omega\to\infty$. This is a deterministic finite-candidate statement; Section 4 gives its statistical and post-selection interpretation.

\begin{definition}[Limiting candidate score]\label{def:limiting-score}
Define the zero-rank score by
\begin{equation}\label{eq:zero-score}
    S_0(Y)=\norm{Y}_F^2-mn\tau^2.
\end{equation}
For $h=1,\ldots,k$, define
\begin{equation}\label{eq:limiting-score}
    S_h(Y)
    =-mn\tau^2+
    \sum_{i=h+1}^{k}\sigma_i^2
    +2\tau^2\left\{|m-n|h+h
    +2\sum_{i=1}^{h}\sum_{\substack{j=1\\j\neq i}}^{k}
        \frac{\sigma_i^2}{\sigma_i^2-\sigma_j^2}
    \right\},
\end{equation}
with the convention that an empty sum is zero. The score $S_0$ is the pointwise limit of the fixed-threshold smooth SURE formula for any realized numerical threshold larger than $\sigma_1$; it represents the candidate that discards all singular components.
\end{definition}

For $h=0$, use the zero-rank candidate $\lambda_0$ introduced above. For $h\geq1$, this expression is obtained from \eqref{eq:smooth-sure} by replacing $f_{\omega,\lambda_h}(\sigma_i)$ with $\sigma_i\ind\{i\leq h\}$ and $f'_{\omega,\lambda_h}(\sigma_i)$ with $\ind\{i\leq h\}$, which is valid pointwise as $\omega\to\infty$ because $\lambda_h\notin\{\sigma_1,\ldots,\sigma_k\}$.

\begin{lemma}[Pointwise convergence on candidates]\label{lem:pointwise}
For every $h=0,1,\ldots,k$,
\begin{equation}
    \operatorname{SURE}_{\omega,
    \lambda_h}(Y)\longrightarrow S_h(Y)
    \qquad\text{as }\omega\to\infty.
\end{equation}
\end{lemma}

\begin{proof}
\emph{Avoiding the discontinuity.} For $h=0$, choose any realized numerical threshold $\lambda_0>\sigma_1$, so that all singular components are discarded in the limiting rule. For $h\geq1$, $\lambda_h$ lies strictly between adjacent singular values, or below $\sigma_k$ when $h=k$. Hence every difference $\sigma_i-\lambda_h$ is nonzero. Therefore
\begin{equation}
    f_{\omega,\lambda_h}(\sigma_i)\to \sigma_i\ind\{i\leq h\},
    \qquad
    f'_{\omega,\lambda_h}(\sigma_i)\to \ind\{i\leq h\}.
\end{equation}
\emph{Finite-sum limit.} The formula \eqref{eq:smooth-sure} contains only finitely many algebraic operations and finite sums over the fixed singular values. Passing the limit through these finite sums gives \eqref{eq:zero-score} for $h=0$ and \eqref{eq:limiting-score} for $h\geq1$.
\end{proof}

\begin{theorem}[Eventual preservation of strict ordering]\label{thm:ordering}
Let $h,\ell\in\{0,1,\ldots,k\}$. If $S_h(Y)<S_{\ell}(Y)$, then there exists $\Omega_{h\ell}<\infty$ such that for all $\omega>\Omega_{h\ell}$,
\begin{equation}
    \operatorname{SURE}_{\omega,\lambda_h}(Y)
    <\operatorname{SURE}_{\omega,\lambda_{\ell}}(Y).
\end{equation}
Consequently, because the extended candidate set including $h=0$ is finite, if $h^*$ is the unique minimizer of $S_h(Y)$ over $h=0,1,\ldots,k$, then $\lambda_{h^*}$ is also the minimizer of $\operatorname{SURE}_{\omega,\lambda_h}(Y)$ over the corresponding plug-in thresholds for all sufficiently large $\omega$.
\end{theorem}

\begin{proof}
\emph{Pairwise comparison.} By Lemma \ref{lem:pointwise},
\begin{equation}
    \operatorname{SURE}_{\omega,\lambda_h}(Y)-
    \operatorname{SURE}_{\omega,\lambda_{\ell}}(Y)
    \longrightarrow S_h(Y)-S_{\ell}(Y)<0.
\end{equation}
\emph{Finite minimization.} The definition of convergence implies that the left-hand side is negative for all sufficiently large $\omega$. The minimizer statement follows by applying this argument to the finitely many pairs $(h^*,\ell)$ with $\ell\neq h^*$ and taking the maximum of the corresponding thresholds $\Omega_{h^*\ell}$.
\end{proof}

\begin{remark}[Ties]
If several $S_h(Y)$ attain the same minimum, the limiting ordering alone does not select a unique candidate. For computational parsimony one may choose the largest threshold among the minimizers, equivalently the smallest retained rank. This tie-breaking rule affects only exactly tied limiting scores. In numerical implementation, the denominators $\sigma_i^2-\sigma_j^2$ in \eqref{eq:limiting-score} should also be protected when singular values are nearly tied. A concrete rule is to set a tolerance
\begin{equation}
    \varepsilon_{\mathrm{sv}}
    =c_{\mathrm{sv}}\,\epsilon_{\mathrm{mach}}\max\{\sigma_1^2,1\},
\end{equation}
with $c_{\mathrm{sv}}$ a modest safety constant, and to regard a pair as numerically tied whenever $|\sigma_i^2-\sigma_j^2|\leq\varepsilon_{\mathrm{sv}}$. When such a near tie is detected, the affected limiting-score comparisons should be treated as numerically unreliable rather than as reliable evidence for a particular global minimizer. In implementation one should either avoid splitting nearly tied singular values or fall back to a conservative rule, such as choosing the smallest retained rank among candidates whose scores remain distinguishable without near-tie denominators. Exact ties have probability zero under the continuous Gaussian model but can occur up to floating-point precision in patch groups.
\end{remark}

\subsection{Recursive ranking}
The limiting scores can be compared recursively. From \eqref{eq:zero-score} and \eqref{eq:limiting-score}, for $h=0,
\ldots,k-1$,
\begin{equation}\label{eq:increment}
\begin{aligned}
    S_{h+1}(Y)-S_h(Y)
    ={}&-\sigma_{h+1}^{2} \\
    &+2\tau^2\left\{|m-n|+1
    +2\sum_{\substack{j=1\\j\neq h+1}}^{k}
    \frac{\sigma_{h+1}^{2}}{\sigma_{h+1}^{2}-\sigma_j^{2}}
    \right\}.
\end{aligned}
\end{equation}
Thus all candidate scores may be ranked by computing the initial value $S_0(Y)$ and then accumulating the adjacent differences. The selected rank is
\begin{equation}
    \widehat h\in\arg\min_{0\leq h\leq k} S_h(Y),
    \qquad
    \widehat\lambda=\lambda_{\widehat h}.
\end{equation}
After ranking, one may report the corresponding hard-truncated approximation $\widehat X_{\widehat\lambda}$, with $\widehat h=0$ returning the zero matrix. The post-selection interpretation is given in Section 4.

\section{Interpretation of the Limiting Score}

\paragraph{Role of the limiting score.}
The score \eqref{eq:limiting-score} has the same algebraic form as the expression obtained by formally substituting hard thresholding into the spectral divergence formula. Hansen showed that such a hard-threshold expression is not, in general, an unbiased risk estimate for singular value hard thresholding \cite{hansen2017comment}. The present use of \eqref{eq:limiting-score} is therefore deliberately narrower: it is a finite-candidate ranking score obtained as the large-$\omega$ limit of valid fixed-threshold smooth SURE formulae, not a new unbiased risk identity for hard thresholding.

\paragraph{Relation to fixed-rank SURE.}
The same limiting score is also closely related to the fixed-rank reduced-rank estimator. For deterministic $h$, retaining the first $h$ singular components gives the fixed-rank estimator
\begin{equation}
    Y\mapsto U\diagop(\sigma_1,\ldots,\sigma_h,0,\ldots,0)V^T,
\end{equation}
which is distinct from fixed-threshold hard thresholding as a statistical parametrization. Hansen's positive result for fixed-rank reduced-rank estimators therefore does not contradict his negative result for hard thresholding \cite{hansen2017comment}. Algebraically, \eqref{eq:limiting-score} agrees with the fixed-rank reduced-rank score: the retained-retained part of the double sum pairs to $h(h-1)$, giving
\begin{equation}
\begin{aligned}
    S_h(Y)
    ={}&-mn\tau^2+\sum_{i=h+1}^{k}\sigma_i^2 \\
    &+2\tau^2\left\{h(|m-n|+h)
    +2\sum_{i=1}^{h}\sum_{j=h+1}^{k}
        \frac{\sigma_i^2}{\sigma_i^2-\sigma_j^2}\right\}.
\end{aligned}
\end{equation}
This identifies an algebraic connection with fixed-rank SURE, while the threshold-based derivation above gives the interpretation used in this paper.

\paragraph{Why use the smooth-threshold route?}
The smooth-threshold construction is not intended to be a stronger alternative to the fixed-rank degrees-of-freedom theory. Rather, it answers a more specific question. Classical SVD denoising is naturally phrased as a cutoff rule: keep large observed singular values and remove small ones. Soft thresholding has a clean SURE theory, but it changes this truncation principle by shrinking every retained singular value. Hard thresholding preserves the truncation principle, but its discontinuity prevents the formal hard-threshold SURE expression from being unbiased. The smooth route gives a controlled way to stay within a threshold-parametrized family: for every deterministic $\lambda$ and finite $\omega$ the estimator satisfies the Stein regularity assumptions, and the limiting finite-candidate argument then recovers the simple ranking score. Thus the mathematical contribution is modest but precise: it justifies a hard-threshold-looking ranking rule through smooth fixed-threshold SURE, rather than claiming a new unbiased SURE formula for discontinuous or data-selected hard thresholding.

\paragraph{SURE tuning and post-selection risk.}
The construction separates three statements:
\begin{enumerate}[label=(\roman*)]
    \item for fixed deterministic $\lambda$, \eqref{eq:smooth-sure} is unbiased;
    \item for realized finite candidates, the plug-in smooth scores converge to $S_h(Y)$;
    \item strict inequalities among finitely many limiting scores are eventually preserved.
\end{enumerate}
Only the first statement is an unbiased-risk identity. It applies to a deterministic threshold:
\begin{equation}
    \E\operatorname{SURE}_{\omega,\lambda}(Y)
    =\E\norm{\widehat X_{\omega,\lambda}(Y)-X}_F^2.
\end{equation}
Once the midpoint candidates $\lambda_h(Y)$ are computed from the observed singular values, or once a threshold is selected by minimizing the same scores, this identity no longer applies automatically. The divergence of the resulting map would include the dependence of the chosen threshold on $Y$, and the selected score has the usual optimism of model and tuning-parameter selection \cite{efron1986biased,efron2004prediction}. Thus SURE minimization is used here as a tuning rule, not as an unbiased post-selection risk estimator.

\section{Empirical Evaluation}

This section reports three matrix-level checks and four image-level studies. The matrix experiments verify fixed-threshold SURE, illustrate post-selection optimism, and compare the induced rank rule with an oracle and residual-energy matching. The image experiments insert the same rank-selection rule into the original SVD denoising pipeline, evaluate finite-$\omega$ smooth reconstruction, examine noise-level misspecification, and compare the two rank rules on Set12 and BSD68. The empirical contribution is the local rank-rule replacement, not a new patch-search or aggregation architecture.

For the matrix-level experiments, eight Set12 images are normalized to $[0,1]$, overlapping $32\times32$ patches are extracted with stride four, and three query patches are sampled from each image. For each query patch, normalized cross-correlation is computed against all other patches from the same image, and the top $50$ most similar patches are retained for the diagnostic patch-group display; this matching step is separate from the full image-denoising patch search in Algorithm \ref{alg:sure-svd}. In the numerical risk experiments, the clean matrix $X\in\R^{32\times32}$ is the query patch itself, with no rank truncation before adding independent Gaussian noise
\begin{equation}
    Y_r=X+W_r,
    \qquad
    (W_r)_{ij}\stackrel{\mathrm{i.i.d.}}{\sim}N(0,\tau^2),
    \qquad r=1,\ldots,N.
\end{equation}
Losses are measured by the Frobenius error
\begin{equation}
    L(\widehat X;X,Y_r)=\norm{\widehat X(Y_r)-X}_F^2.
\end{equation}
Oracle ranks use the clean matrix $X$ only as a benchmark and are not available to a practical denoising method. Implementation details such as random seeds, grids, tie breaking and near-tie tolerances are recorded in the replication code at \url{https://github.com/ECFDPB/SURE-SVD-Denoising}.

Figure \ref{fig:diagnostic-patch-group} illustrates the patch-group construction used in the matrix-level experiments. It provides setup context rather than serving as one of the three validation plots.
\begin{figure}[t]
\centering
\includegraphics[width=\textwidth]{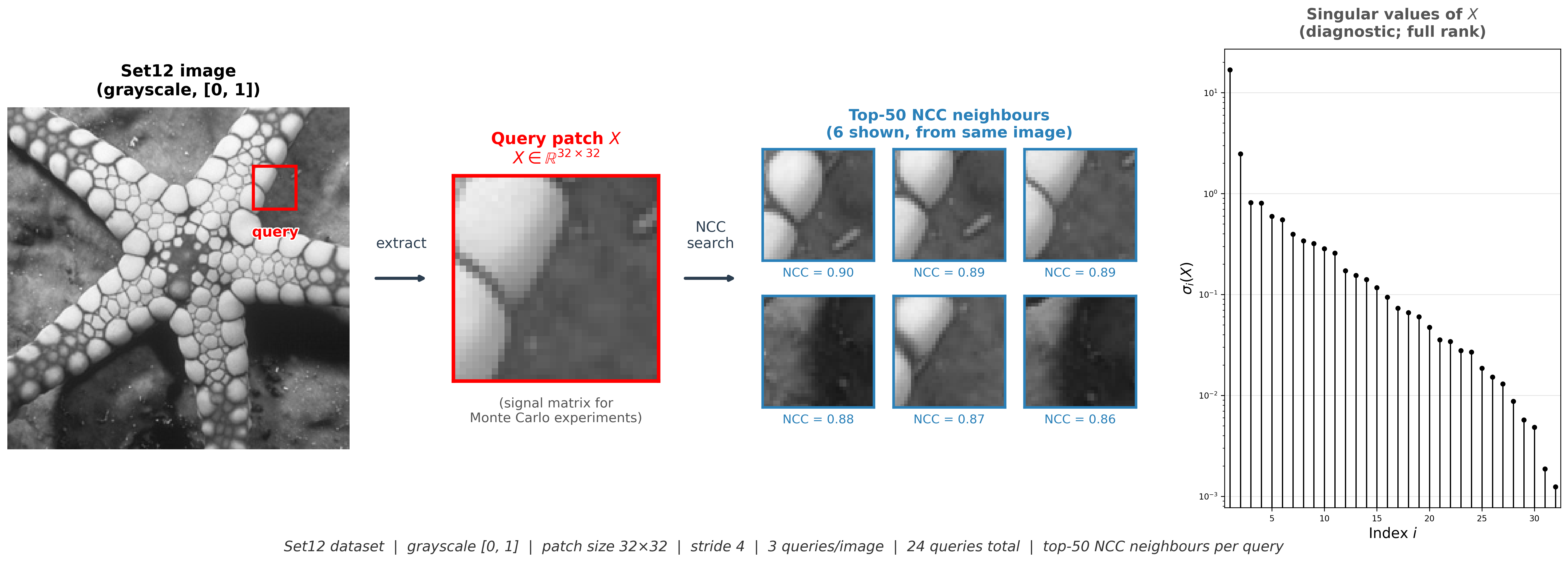}
\caption{Patch construction schematic: a query patch is extracted from a normalized grayscale image, compared with patches from the same image by normalized cross-correlation, and associated with its nearest patch neighbours. Pixel intensities are normalized in this diagnostic display to compare relative eigenvalue magnitudes; this normalization is separate from the $[0,255]$ image-denoising pipeline. The matrix-level experiments below use the query patch itself as the clean matrix $X$.}
\label{fig:diagnostic-patch-group}
\end{figure}

\subsection{Monte Carlo check of fixed-threshold unbiasedness}

The first experiment tests Proposition \ref{prop:sure-unbiased} with fixed deterministic $(\omega,\lambda)$. Over independent replicates we compare
\begin{equation}
    \widehat R_{\omega,
    \lambda}
    =\frac{1}{N}\sum_{r=1}^{N}
    \norm{\widehat X_{\omega,
    \lambda}(Y_r)-X}_F^2,
    \qquad
    \widehat S_{\omega,
    \lambda}
    =\frac{1}{N}\sum_{r=1}^{N}
    \operatorname{SURE}_{\omega,
    \lambda}(Y_r)
\end{equation}
which should agree up to Monte Carlo error.

Figure \ref{fig:exp1-sure-unbiasedness} shows the expected agreement. Panel (a) compares Monte Carlo mean SURE with Monte Carlo mean loss, and panel (b) reports the absolute relative discrepancy
\begin{equation}
    \frac{|\widehat S_{\omega,
    \lambda}-\widehat R_{\omega,
    \lambda}|}{\widehat R_{\omega,
    \lambda}}
\end{equation}
as a function of $\omega$. The discrepancies remain small over the tested range of smoothing parameters, including large $\omega$ where the logistic shrinker
is close to hard thresholding away from the threshold. This is consistent with the theoretical statement. The observed nonzero discrepancies are finite Monte Carlo errors, not evidence of systematic bias.

\begin{figure}[t]
\centering
\begin{minipage}[t]{0.47\textwidth}
\centering
\includegraphics[width=\textwidth]{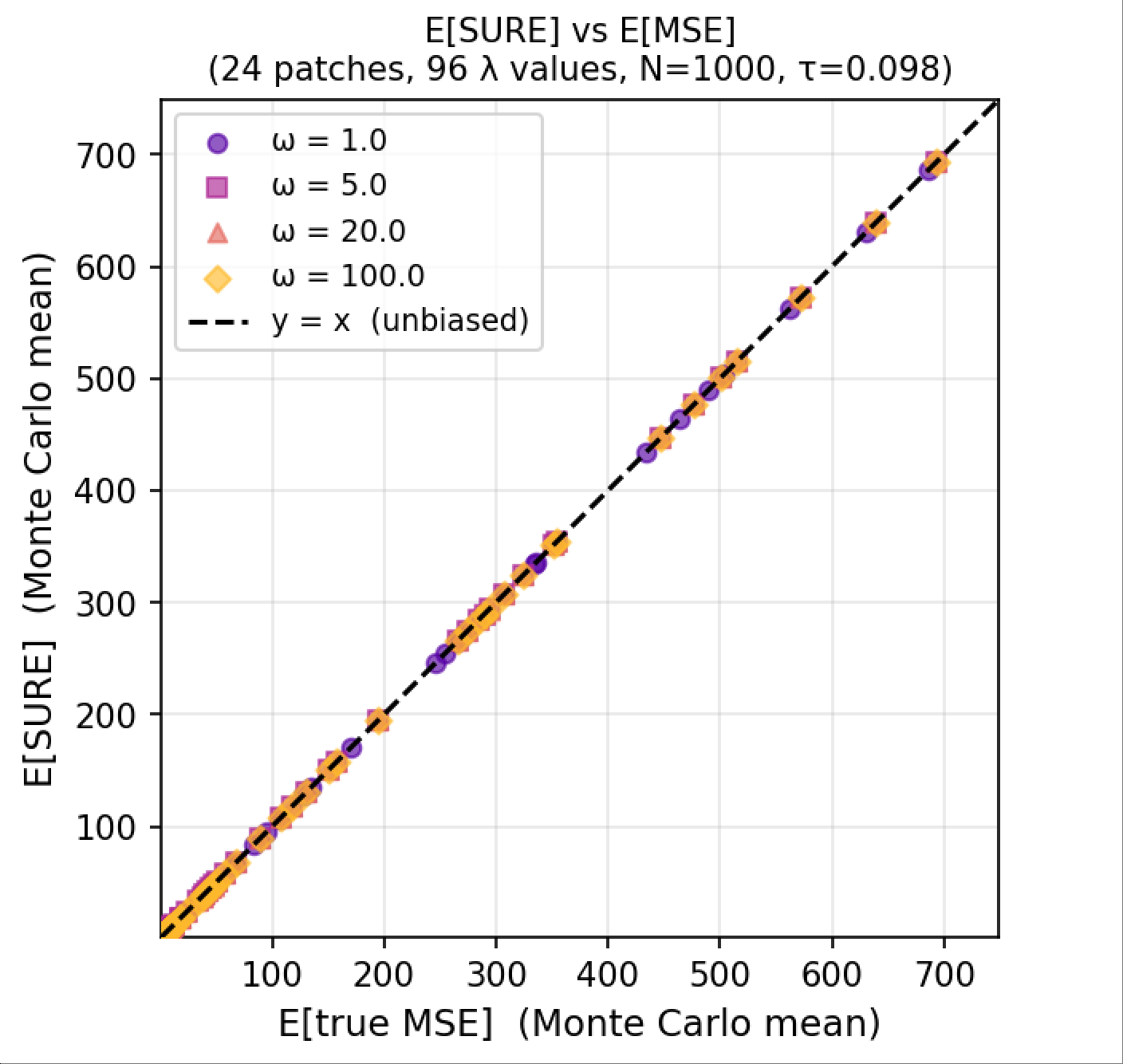}
\par\smallskip
{\small (a) Monte Carlo mean SURE versus Monte Carlo mean true loss for fixed $(\lambda,\omega)$.}
\end{minipage}\hfill
\begin{minipage}[t]{0.47\textwidth}
\centering
\includegraphics[width=\textwidth]{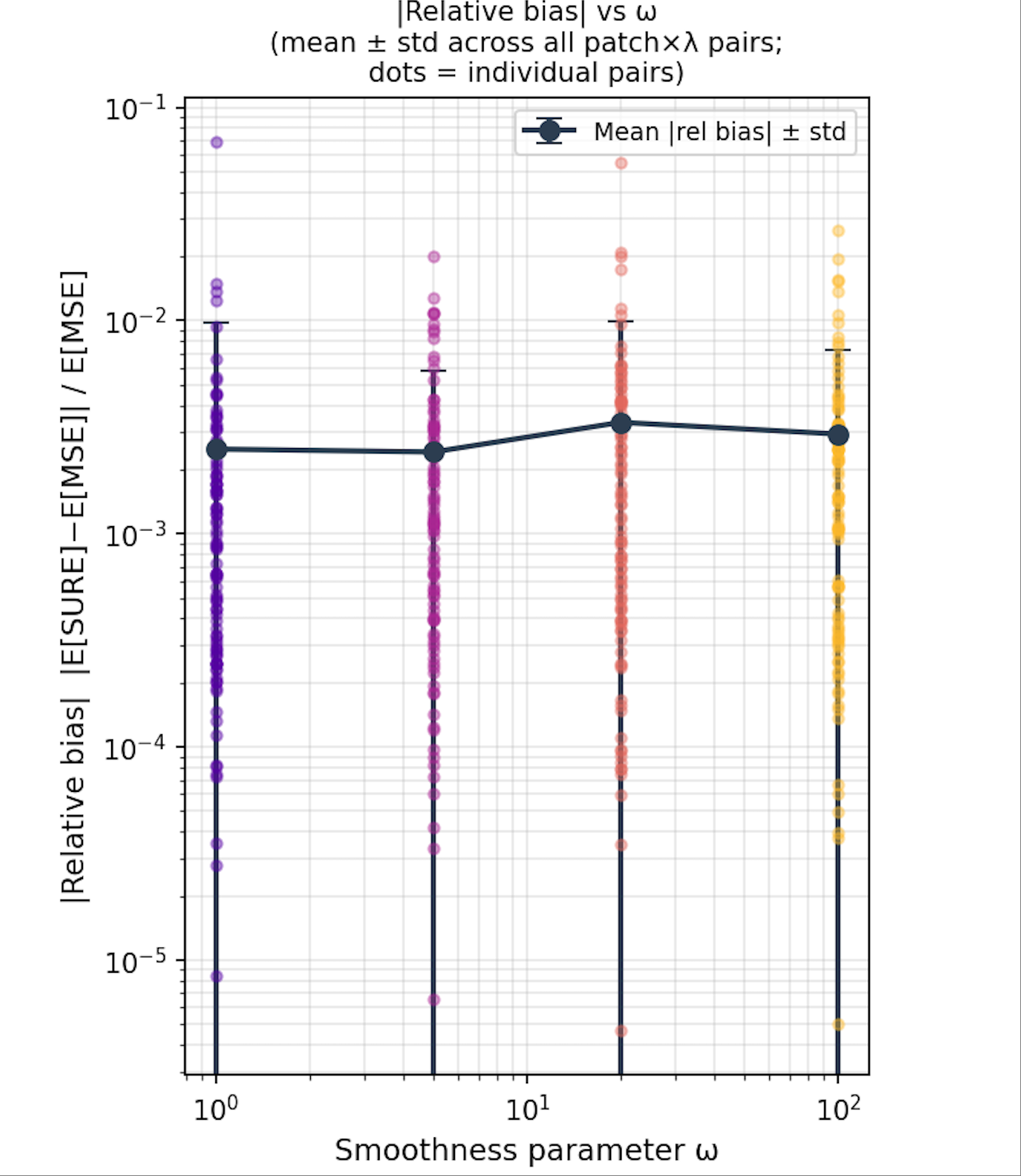}
\par\smallskip
{\small (b) Absolute relative discrepancy as a function of the smoothness parameter $\omega$.}
\end{minipage}
\caption{Monte Carlo verification of the fixed-threshold identity in Proposition \ref{prop:sure-unbiased}. The thresholds are deterministic in this experiment, so the setup matches the hypotheses of the SURE theorem. The agreement between the Monte Carlo mean SURE and the Monte Carlo mean true loss confirms the implementation of \eqref{eq:smooth-sure} up to simulation error.}
\label{fig:exp1-sure-unbiasedness}
\end{figure}

\subsection{Post-selection optimism and low-risk rank selection}

The second experiment illustrates the caveat in Section 4. For each noisy matrix, the limiting scores are minimized over a finite rank set $\mathcal H$ to select rank:
\begin{equation}
    \widehat h_{\mathrm{SURE}}(Y_r)
    \in\arg\min_{h\in\mathcal H} S_h(Y_r).
\end{equation}
The selected score $S_{\widehat h_{\mathrm{SURE}}}(Y_r)$ is compared with the true loss of the selected estimator and with the oracle rank
\begin{equation}
    h_{\mathrm{or}}(Y_r)
    \in\arg\min_{h\in\mathcal H} L(\widehat X_h;X,Y_r),
\end{equation}
where $\widehat X_h$ is the hard rank-$h$ truncation. This oracle is used only to measure how much excess loss is incurred by the data-driven rank selection.

Figure \ref{fig:exp2-postselection} shows both effects: the selected SURE score is optimistic as a risk estimate, but the SURE-selected rank has true loss close to the oracle. Thus the score is used as a tuning criterion, not as an unbiased post-selection risk estimate.

\begin{figure}[t]
\centering
\begin{minipage}[t]{0.43\textwidth}
\centering
\includegraphics[width=\textwidth]{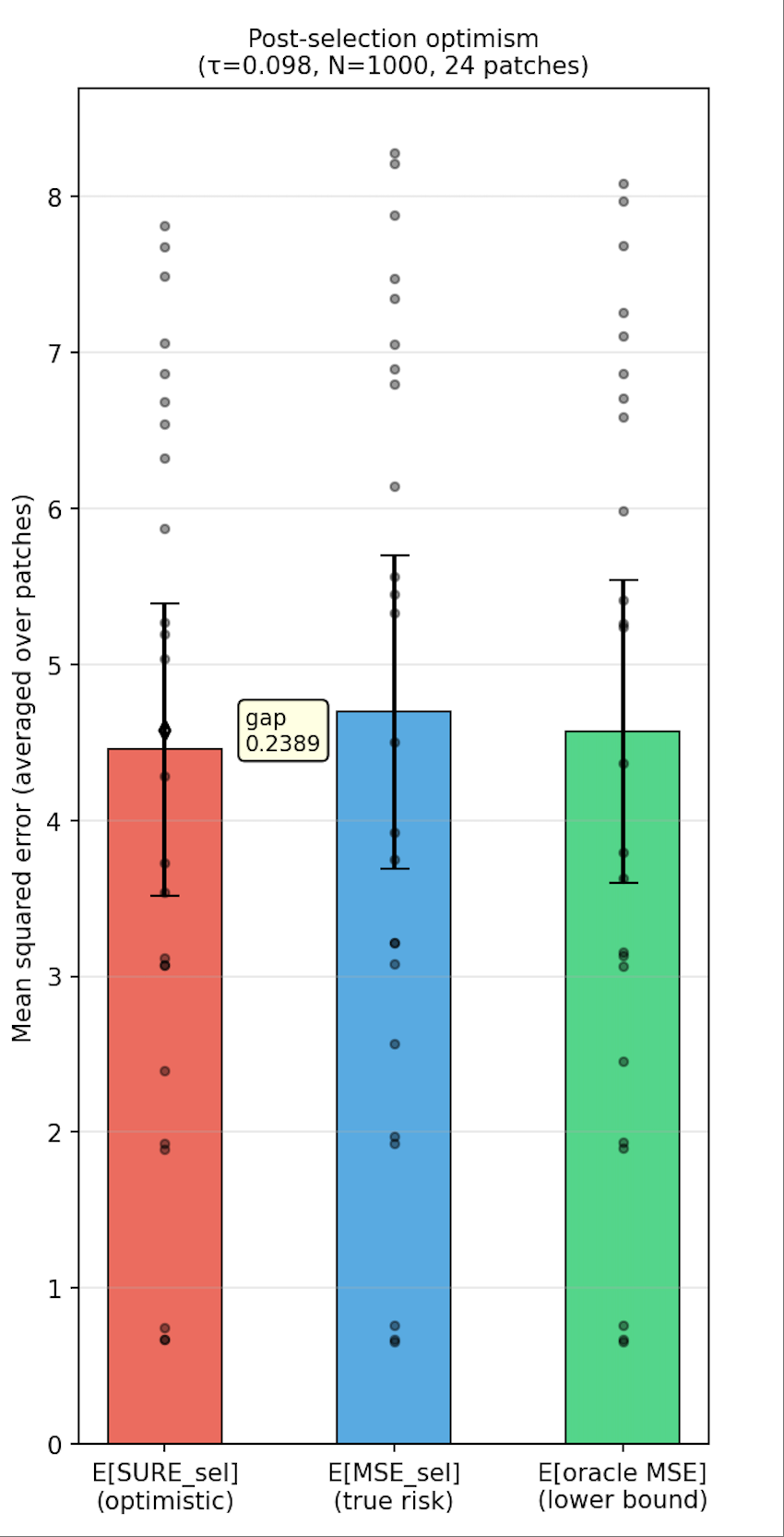}
\par\smallskip
{\small (a) Selected SURE score, selected true loss, and oracle loss.}
\end{minipage}\hfill
\begin{minipage}[t]{0.47\textwidth}
\centering
\includegraphics[width=\textwidth]{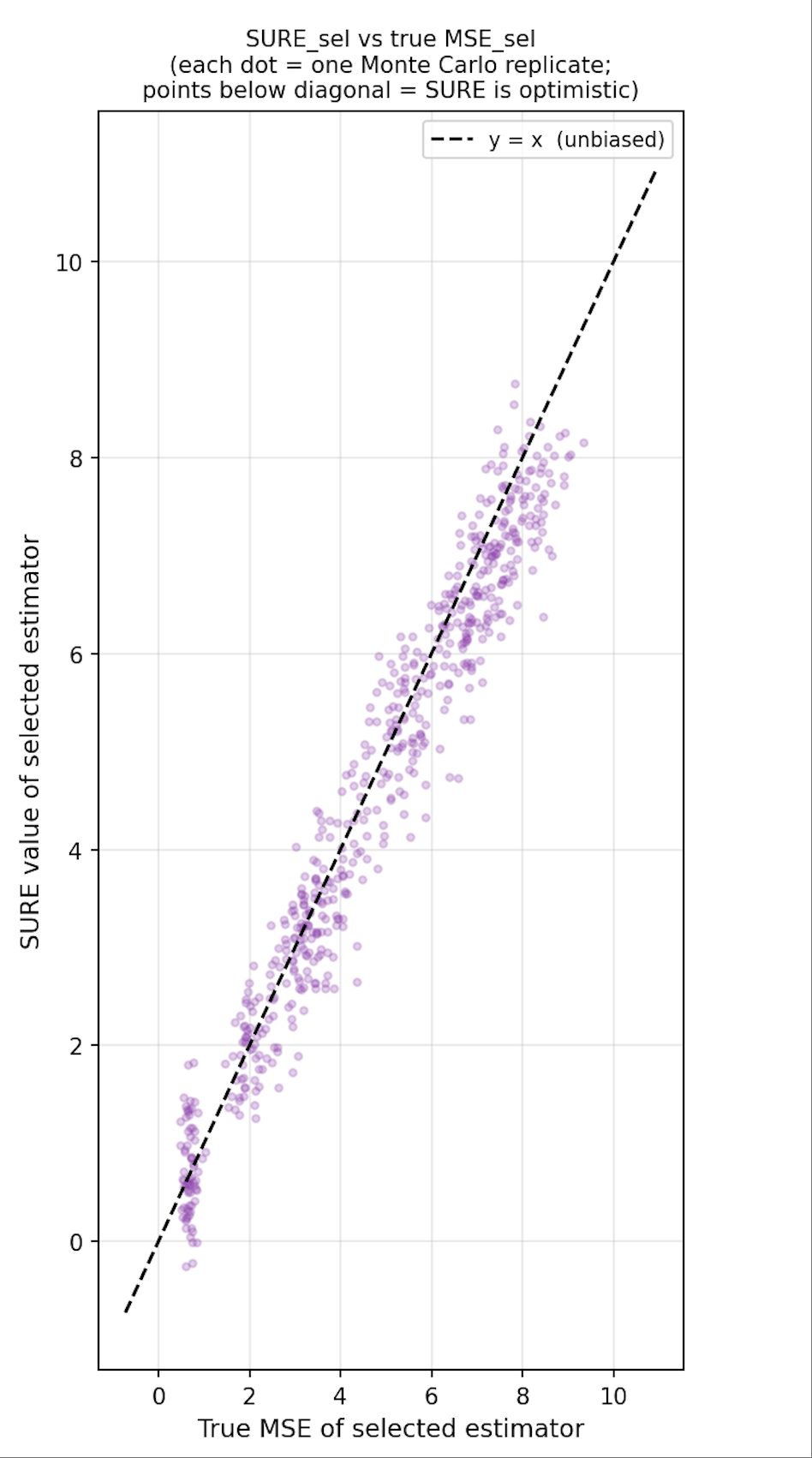}
\par\smallskip
{\small (b) Replicate-level selected SURE score versus selected true loss.}
\end{minipage}
\caption{Post-selection behaviour of the plug-in ranking step. The selected SURE value is optimistic after minimization over candidate ranks, as shown by the gap between selected SURE and selected true loss and by the scatter plot lying below the diagonal. At the same time, the true loss of the SURE-selected rank is close to the oracle loss, showing that the ranking rule can still choose a low-risk rank.}
\label{fig:exp2-postselection}
\end{figure}

\subsection{Oracle-style comparison with energy matching}

The third experiment compares oracle, SURE and residual-energy rank selection over several noise levels $\tau$. The oracle rule chooses the best rank using $X$ and is included only as a benchmark.

For a method $M$, define its empirical oracle efficiency at noise level $\tau$ by
\begin{equation}
    \operatorname{Eff}_{\tau}(M)
    =\frac{\widehat{\E}_{\tau}\,L(\widehat X_{h_{\mathrm{or}}};X,Y)}
    {\widehat{\E}_{\tau}\,L(\widehat X_{h_M};X,Y)}.
\end{equation}
Values close to one indicate near-oracle performance, while smaller values indicate larger excess risk.

Figure \ref{fig:exp3-oracle-comparison} shows that SURE closely tracks the oracle curve and has higher oracle efficiency than residual-energy matching, consistent with the degrees-of-freedom correction in the SURE score.

\begin{figure}[t]
\centering
\begin{minipage}[t]{0.48\textwidth}
\centering
\includegraphics[width=\textwidth]{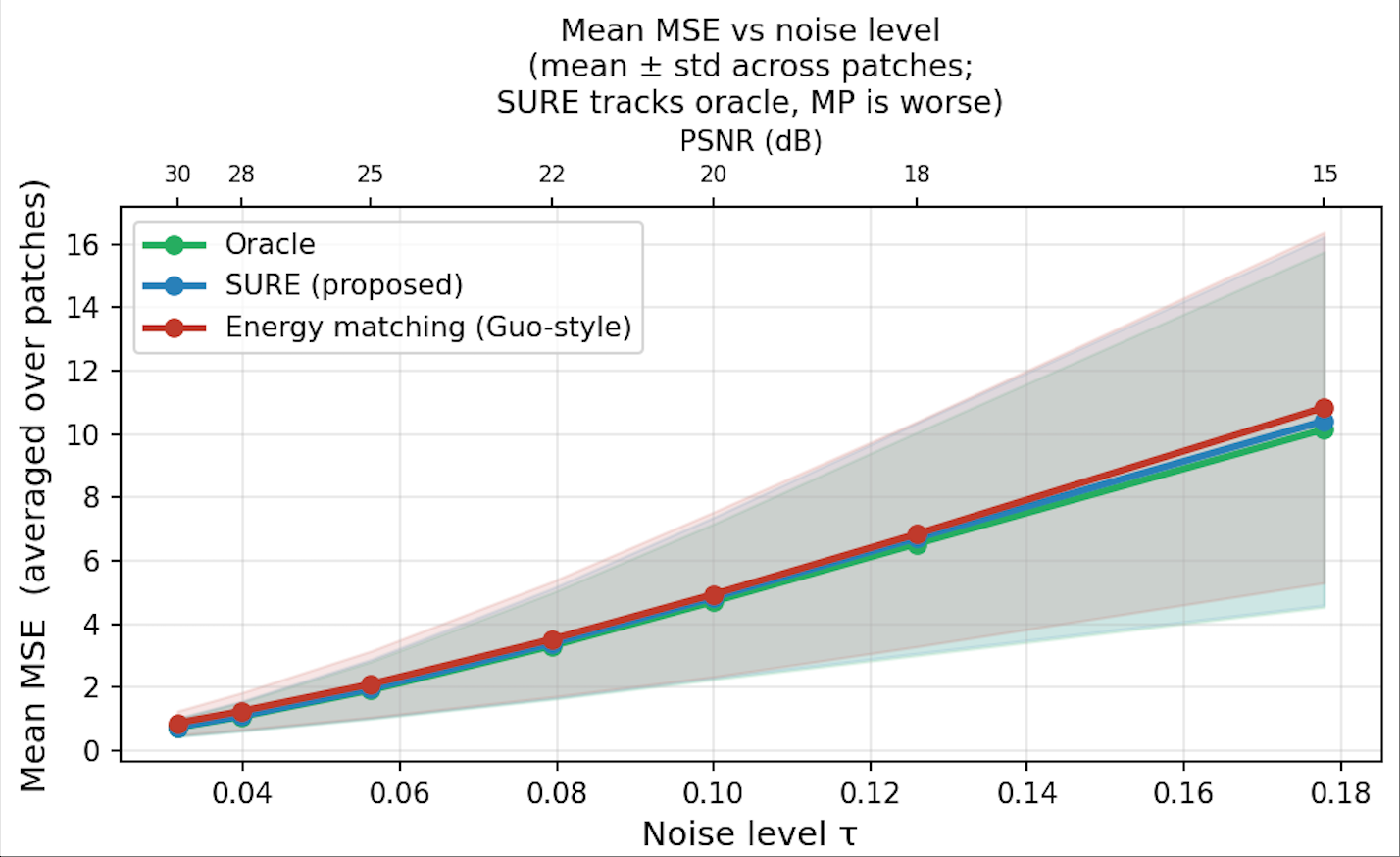}
\par\smallskip
{\small (a) Mean MSE as a function of the noise standard deviation $\tau$.}
\end{minipage}\hfill
\begin{minipage}[t]{0.48\textwidth}
\centering
\includegraphics[width=\textwidth]{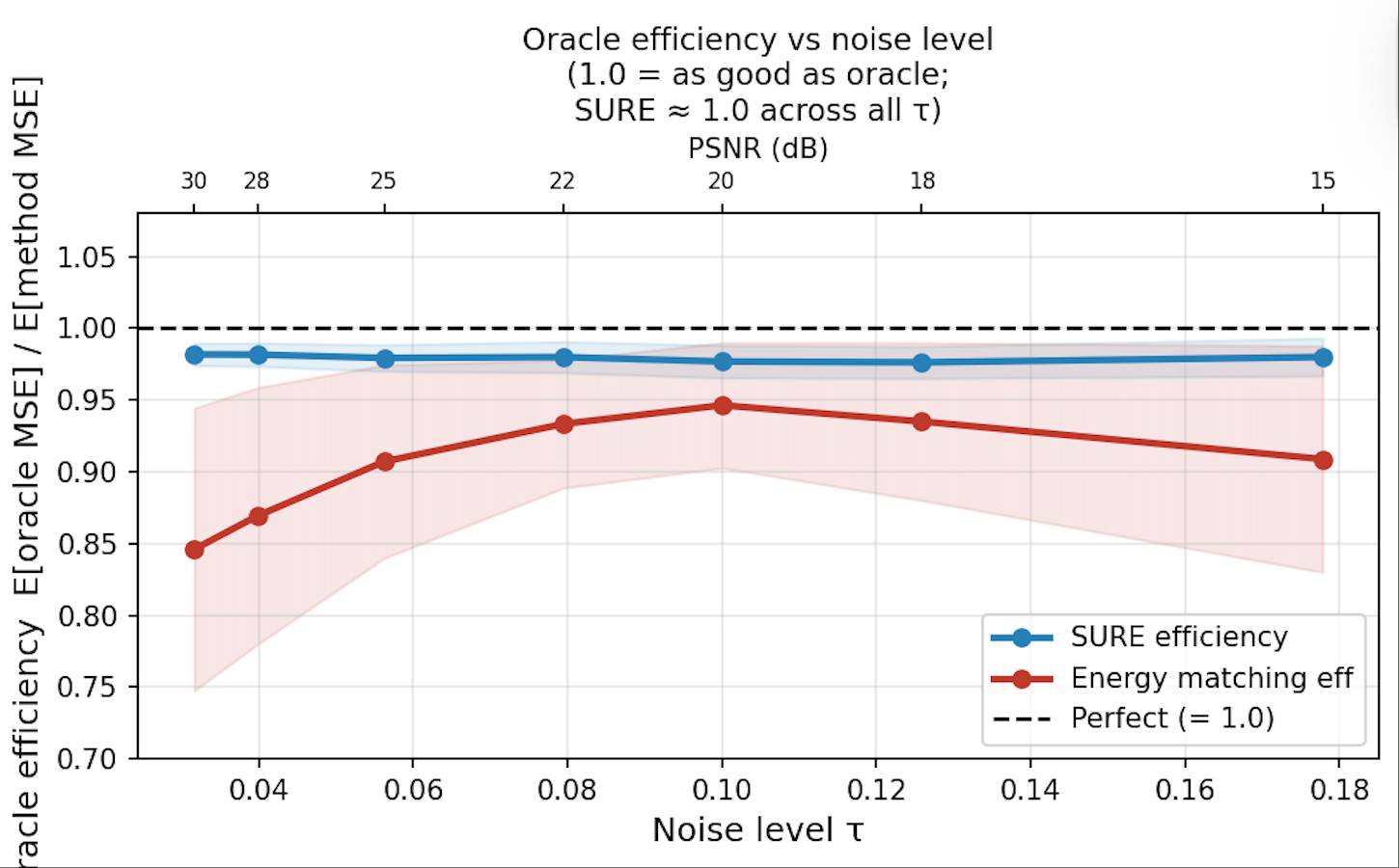}
\par\smallskip
{\small (b) Oracle efficiency as a function of $\tau$.}
\end{minipage}
\caption{Oracle-style comparison of rank-selection rules. The proposed SURE rule remains close to the oracle benchmark across the tested noise levels, while the residual-energy matching baseline has larger excess risk and lower oracle efficiency. The oracle is not implementable in practice; it is used only to quantify the excess loss of the data-driven rules.}
\label{fig:exp3-oracle-comparison}
\end{figure}

\subsection{Complete SVD denoising pipelines}

The preceding experiments isolate rank selection for one noisy matrix. The image-denoising experiment applies the same local SVD truncation primitive to patch-group matrices. In the synthetic setting one observes
\begin{equation}
    Y=X+W,
    \qquad W_{ab}\stackrel{\mathrm{i.i.d.}}{\sim}N(0,\tau^2),
\end{equation}
computes the SVD of $Y$, selects a rank $h$, and returns the hard rank-$h$ truncation. In the image setting this operation is applied to group matrices $P_j\in\R^{m\times n_j}$ formed from similar patches, where $m$ is the number of pixels in a vectorized patch and $n_j$ is the number of patches in group $j$. In the experiments each group consists of one reference patch and $L$ matched patches, so $n_j=L+1$. This is an algorithmic transfer, not a direct application of the fixed deterministic-matrix model, because grouping, overlap, aggregation and back projection are data-dependent pipeline operations.

The original method of Guo et al.\ \cite{guo2016efficient} selects the smallest rank whose discarded singular-value energy is at most $m n_j\tau^2$. The SURE-modified method instead evaluates the scores \eqref{eq:zero-score}--\eqref{eq:limiting-score} for $Y=P_j$ and dimensions $(m,n_j)$, selects
\begin{equation}
    \widehat h_j\in\arg\min_{0\leq h\leq k_j}S_h(P_j).
\end{equation}
and returns the rank-$\widehat h_j$ hard truncation. Thus the only algorithmic substitution is the local rank-selection rule; Algorithm \ref{alg:sure-svd} gives the complete implementation.

After aggregating the first-pass denoised groups into $\widehat{\mathbf x}_0$, the back-projection step feeds a controlled fraction of the observed residual back into the estimate:
\begin{equation}
    \widetilde{\mathbf y}
    =\widehat{\mathbf x}_0
    +\delta(\mathbf y-\widehat{\mathbf x}_0),
    \qquad \delta=0.5.
\end{equation}
Thus $\widetilde{\mathbf y}$ lies between the first-pass estimate and the noisy observation; the purpose of this residual-feedback step is to restore signal detail that may have been removed in the first pass before repeating patch grouping and SVD denoising. In this identity-observation setting, ``back projection'' refers to this residual correction rather than inversion of a nontrivial forward operator. The second pass uses the updated noise level specified in Algorithm \ref{alg:sure-svd}.

Figure \ref{fig:sure-modified-pipeline} and Algorithm \ref{alg:sure-svd} summarize the modified pipeline and its implementation-level steps.

\clearpage

\begin{figure}[H]
\centering
\includegraphics[width=0.60\textwidth]{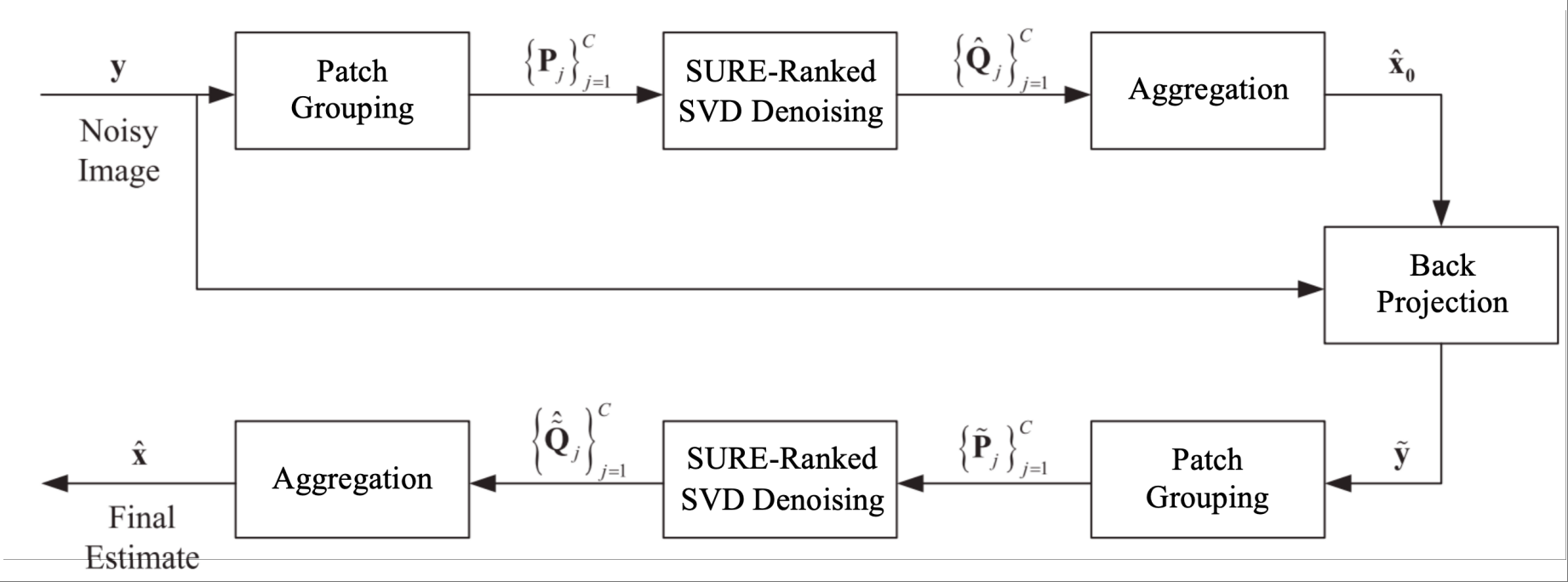}
\caption{SURE-modified two-pass SVD denoising pipeline. The method uses the original patch grouping, SVD truncation, aggregation, back projection, and refinement structure; only the local rank-selection rule inside each grouped-patch SVD block is changed.}
\label{fig:sure-modified-pipeline}
\end{figure}

\vspace{-0.4em}
\input{figures/algorithm_box.tex}

\clearpage

\subsection{Auxiliary finite-$\omega$ reconstruction study}
\label{sec:finite-omega-denoising}

The finite-$\omega$ estimator is not the final reconstruction rule in the proposed practical pipeline. Its role is to provide a smooth fixed-threshold estimator for which Stein's formula is valid; the limit $\omega\to\infty$ yields the candidate-ranking score, after which the selected rank is implemented by hard singular-value truncation. Thus, $\omega$ is not an additional tuning parameter of the final procedure.

To assess whether the smooth estimator itself can also produce meaningful reconstructions, we include an auxiliary experiment in which finite values of $\omega$ are deliberately used for denoising. This study illustrates the behavior of the smooth family and its approach to the hard-limit output; it does not modify the proposed practical algorithm.

After selecting the candidate threshold by the same SURE-based rule, each singular value $s$ was reconstructed as
\begin{equation}
    f_{\omega,\lambda}(s)
    =\frac{s}{1+\exp[-\omega(s-\lambda)]},
\end{equation}
rather than by hard truncation. Consequently, each finite-$\omega$ entry in this experiment represents a complete smooth-threshold reconstruction, not merely the evaluation of a smooth criterion followed by hard thresholding. As elsewhere in the paper, the clean image was used only for computing PSNR and SSIM and was not used in selecting $\lambda$.

The smoothness values are reported for intensities normalized to $[0,1]$. In the implementation on the 8-bit scale, if $s_{255}=255s$ and $\lambda_{255}=255\lambda$, the corresponding parameter is $\omega_{255}=\omega/255$, so that
\begin{equation}
    \omega_{255}(s_{255}-\lambda_{255})
    =\omega(s-\lambda).
\end{equation}
Thus, this conversion is only a change of units and does not introduce an additional tuning rule. The case $\omega=\infty$ denotes the hard-limit SURE selection rule together with hard singular-value truncation.

Table \ref{tab:finite-omega-denoising} reports average PSNR and SSIM over the 12 Set12 images at three noise levels. For each image and noise level, all values of $\omega$ were evaluated using the same noisy observation.

\begin{table}[H]
\centering
\caption{Auxiliary denoising results for the finite-$\omega$ smooth estimator on Set12. Entries are average PSNR in dB followed by average SSIM. The row $\omega=\infty$ is the hard-threshold output used in the proposed practical pipeline.}
\label{tab:finite-omega-denoising}
{\scriptsize
\setlength{\tabcolsep}{3.0pt}
\renewcommand{\arraystretch}{0.92}
\begin{tabular}{c|ccc}
\toprule
$\omega$ & $\sigma=10$ & $\sigma=30$ & $\sigma=50$ \\
\midrule
$1$ & $33.097/0.9085$ & $28.156/0.8064$ & $25.828/0.7420$ \\
$5$ & $34.214/0.9226$ & $28.906/0.8266$ & $26.346/0.7586$ \\
$20$ & $34.311/0.9255$ & $28.956/0.8282$ & $26.340/0.7582$ \\
$100$ & $34.237/0.9241$ & $28.916/0.8267$ & $26.310/0.7571$ \\
$1000$ & $34.204/0.9237$ & $28.915/0.8267$ & $26.310/0.7571$ \\
$\infty$ & $34.204/0.9237$ & $28.915/0.8267$ & $26.310/0.7571$ \\
\bottomrule
\end{tabular}
}
\end{table}

Figure \ref{fig:finite-omega-summary} reports the corresponding differences from the hard-limit output.

\begin{figure}[H]
\centering
\includegraphics[width=\textwidth]{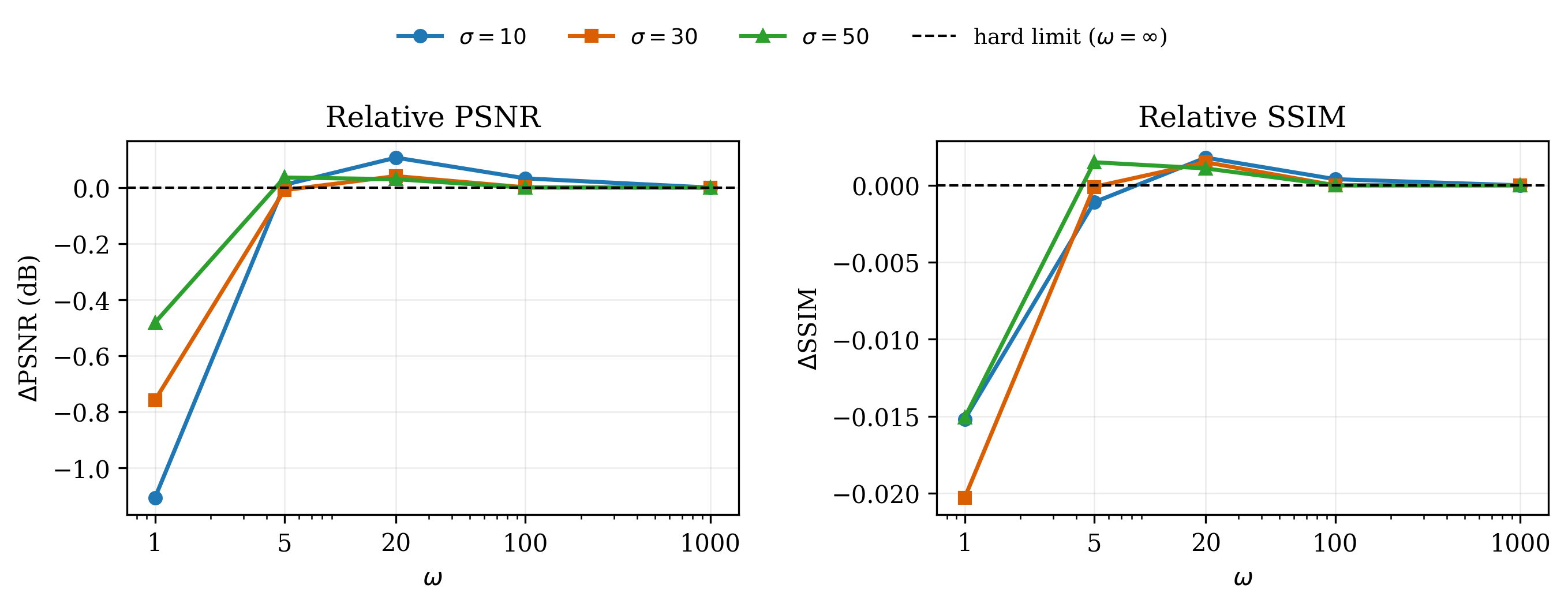}
\caption{Auxiliary comparison of finite-$\omega$ smooth reconstructions with the hard-limit output used by the proposed practical pipeline. The curves show differences in average PSNR and SSIM relative to $\omega=\infty$, so the horizontal zero line represents the final hard-threshold reconstruction. This experiment illustrates convergence toward the hard limit and is not used to tune an additional algorithmic parameter.}
\label{fig:finite-omega-summary}
\end{figure}

These results should not be interpreted as introducing $\omega$ as a parameter to be selected in the proposed pipeline. They show that the gradual transition at $\omega=1$ causes noticeable oversmoothing, whereas the reconstructions rapidly approach the hard-limit performance as $\omega$ increases. The small positive differences observed for some intermediate values are descriptive finite-sample variations, not evidence for a new optimal setting. In particular, the results for $\omega=100$ and $\omega=1000$ are essentially indistinguishable from hard truncation.

Accordingly, the main role of the smooth family remains theoretical: it provides a regular fixed-threshold estimator to which Stein's identity applies and, through its hard limit, yields the ranking score used by the practical method. The final reconstructions reported elsewhere in this paper continue to use hard singular-value truncation.

\subsection{Sensitivity to noise-level misspecification}
\label{sec:noise-level-sensitivity}

The SURE criterion requires a supplied noise standard deviation. In practice, this quantity is generally estimated rather than known exactly. We therefore examined the effect of replacing the true noise level $\sigma$ by
\begin{equation}
    \sigma_{\mathrm{est}}=r\sigma,
    \qquad
    r\in\{0.8,0.9,1,1.1,1.2\}.
\end{equation}
Thus, the experiment covers relative noise-level errors between $-20\%$ and $+20\%$.

We used the 12 Set12 images and true noise levels $\sigma\in\{10,30,50\}$. For every image and noise level, three independent Gaussian-noise realizations were generated, giving 36 matched observations for each pair $(\sigma,r)$. Within each matched observation, the noisy image was held fixed while only the value supplied to the denoising algorithm was changed. The estimated noise level was used throughout the complete two-pass pipeline, including both the SURE score and the second-stage noise update. The clean image was again used only for evaluation.

The patch-size regime was kept unchanged across all values of $r$ at a fixed true noise level. Consequently, the reported changes primarily reflect the use of a misspecified noise level in the SURE-based rank selection and noise update, rather than a discrete change in patch size.

Table \ref{tab:noise-level-sensitivity} gives the mean paired change relative to the correctly specified case $r=1$. The corresponding baseline mean PSNR/SSIM values at $r=1$ are $34.211/0.9235$, $28.895/0.8262$, and $26.324/0.7575$ for $\sigma=10,30,50$, respectively.

\begin{table}[t]
\centering
\caption{Sensitivity to noise-level misspecification on Set12. Entries are mean paired changes relative to using the true noise level, averaged over 36 matched observations. Positive values indicate improvement relative to $r=1$.}
\label{tab:noise-level-sensitivity}
{\scriptsize
\setlength{\tabcolsep}{3.0pt}
\renewcommand{\arraystretch}{0.92}
\begin{tabular}{cc|rrrr}
\toprule
True $\sigma$ & Metric & $-20\%$ & $-10\%$ & $+10\%$ & $+20\%$ \\
\midrule
$10$ & $\Delta$PSNR (dB) & $-2.045$ & $-0.479$ & $+0.044$ & $-0.154$ \\
$10$ & $\Delta$SSIM & $-0.0216$ & $-0.0005$ & $-0.0033$ & $-0.0084$ \\
\midrule
$30$ & $\Delta$PSNR (dB) & $-1.057$ & $+0.171$ & $-0.301$ & $-0.603$ \\
$30$ & $\Delta$SSIM & $-0.0094$ & $+0.0068$ & $-0.0077$ & $-0.0148$ \\
\midrule
$50$ & $\Delta$PSNR (dB) & $+0.001$ & $+0.301$ & $-0.297$ & $-0.567$ \\
$50$ & $\Delta$SSIM & $+0.0060$ & $+0.0092$ & $-0.0083$ & $-0.0154$ \\
\bottomrule
\end{tabular}
}
\end{table}

\begin{figure}[t]
\centering
\includegraphics[width=\textwidth]{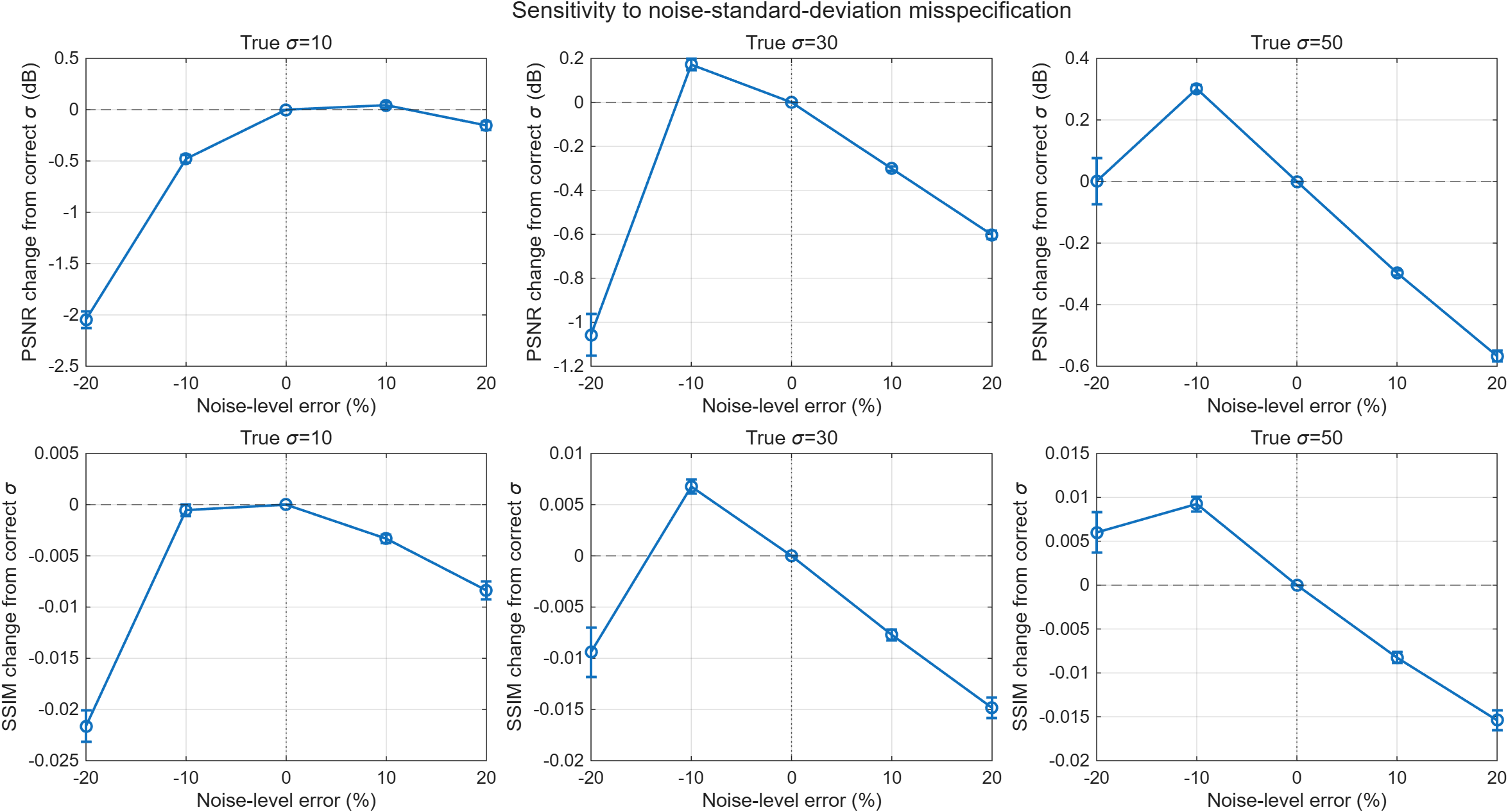}
\caption{Sensitivity of the complete denoising pipeline to errors in the supplied noise level. The curves show mean paired changes in PSNR and SSIM relative to the correctly specified case $\sigma_{\mathrm{est}}=\sigma$. Error bars are standard errors over the 36 matched observations.}
\label{fig:noise-level-sensitivity}
\end{figure}

Around the correctly specified noise level, moderate misspecification generally produces small-to-moderate changes in reconstruction quality. The response is nevertheless asymmetric and need not be monotone. This behavior is expected because the estimated noise level enters both the SURE penalty and the second-stage noise update. Moreover, rank selection is discrete: a small change in $\sigma_{\mathrm{est}}$ may change the selected rank of some patch groups, even though the change in the supplied noise level is continuous.

The signal-to-noise regime also affects this behavior. At high SNR, the dominant signal singular values are generally better separated from the noise singular values, so the selected ranks tend to remain stable under modest perturbations of the noise estimate. At lower SNR, signal and noise singular values overlap more strongly. Consequently, a change in $\sigma_{\mathrm{est}}$ can move more singular values across the effective selection boundary, producing more visible variation across underestimation and overestimation. Such variation is therefore a natural consequence of spectral ambiguity in the low-SNR regime, rather than evidence of numerical instability.

There is, however, an important exception to this general tendency. A substantial underestimation of the noise level can be harmful even at high SNR. In particular, using $0.8\sigma$ when the true noise level is $\sigma=10$ decreases average PSNR by $2.045$ dB. Underestimation weakens the SURE penalty and can retain too many noise-dominated singular components; because the correctly denoised high-SNR image already has a small residual error, the remaining noise can produce a comparatively large PSNR loss. Overestimation has the complementary effect of discarding weak signal components.

Overall, the pipeline is reasonably stable under modest noise-level errors, but it is not globally insensitive to noise-level misspecification. The fluctuations observed in the lower-SNR cases are consistent with the reduced separation between signal and noise singular values, while substantial underestimation should be avoided in all noise regimes.

\subsection{Paired BSD68 comparison and Wilcoxon tests}

To assess whether this pattern generalizes across images, we compare the two controlled SVD pipelines on BSD68 \cite{martin2001database}. For each image and noise level in the recorded result file, the same noisy realization is reused by both rank-selection rules. Thus the paired unit is the image at a fixed noise level, and the comparison isolates the effect of replacing residual-energy matching by SURE ranking. For PSNR and SSIM we apply a one-sided Wilcoxon signed-rank test to the paired differences
\[
    \Delta=\text{SURE-modified SVD}-\text{energy-matching SVD}
\]
at each noise level, testing for a positive signed-rank location shift of the paired-difference distribution.

\input{figures/table_bsd68_wilcoxon.tex}

Table \ref{tab:bsd68-wilcoxon} shows that the SURE modification is not uniformly beneficial. At the low-noise setting $\sigma=10$, the paired differences are negative, so the one-sided tests give no evidence for a SURE advantage. At $\sigma=30$ the two rules are close, with slightly negative average differences. At $\sigma=50$, however, SURE gives positive average gains in both PSNR and SSIM, with win rates of $65\%$ and $68\%$ and nominal one-sided Wilcoxon $p$-values below $0.05$. The runtime overhead is small in this implementation: averaged over the reported BSD68 runs, the SURE-modified pipeline takes $84.65$ seconds per image versus $83.60$ seconds for energy matching, about a $1.3\%$ increase. These results support a noise-regime interpretation: the SURE correction is most useful in the high-noise setting, rather than a uniformly dominant replacement for energy matching.

\subsection{Illustrative comparison with SVD-related methods}

We finally give an illustrative fixed-realization comparison involving K-SVD \cite{elad2006image}, LPG-PCA \cite{zhang2010two}, the original energy-matching SVD pipeline, and the SURE-modified SVD pipeline on Set12. The displayed methods are all connected to the singular-value-decomposition viewpoint: K-SVD uses SVD in its dictionary-update step, LPG-PCA uses local PCA, and the two SVD pipelines use local rank-truncated SVD. The tested 8-bit noise levels are
\begin{equation}
    \sigma\in\{10,30,50\}.
\end{equation}
These three levels are used as representative low, moderate, and high noise regimes. The reported experiment was implemented in MATLAB for noise generation. For image index $i$ and noise level $\sigma$, the script calls \texttt{rng($100i+\sigma$)} immediately before MATLAB's \texttt{randn}, and then forms
\begin{equation}
    y=x+\sigma Z,
    \qquad Z_{ab}\stackrel{\mathrm{i.i.d.}}{\sim}N(0,1).
\end{equation}
The seed formula identifies the released MATLAB implementation and is not intended as a language-independent specification of the Gaussian realization. For each image--noise pair, the resulting noise array is held fixed and all methods use the same single noisy realization. K-SVD uses Ron Rubinstein's MATLAB implementation (\texttt{ksvdbox13}, \texttt{ompbox10}) with parameters
\begin{equation}
\begin{gathered}
    \texttt{blocksize}=8,
    \quad \texttt{dictsize}=256,
    \quad \texttt{sigma}=\sigma,
    \quad \texttt{maxval}=255, \\
    \texttt{trainnum}=40000,
    \quad \texttt{iternum}=10,
    \quad \texttt{memusage}=\texttt{'high'}.
\end{gathered}
\end{equation}
LPG-PCA uses the official MATLAB code of Zhang et al. with profile \texttt{'fast'} and $K=0$. The energy-matching and SURE columns are the controlled comparison of interest: they use identical Python pipeline code and differ only in the rank-selection function. K-SVD and LPG-PCA are included only to provide context for the fixed-realization image-denoising scale within the family of SVD- or PCA-related methods. The shared parameters of the two SVD pipelines are patch size $9\times9$, $10\times10$, or $11\times11$ according to noise level; $L=85$ similar patches; search-window half-size $35$; reference-patch stride $3$; back-projection parameter $\delta=0.5$; and, in the experiments, noise-update scaling factor $\gamma=0.65$. The full implementation and scripts are available at \url{https://github.com/ECFDPB/SURE-SVD-Denoising}.

PSNR is computed on the full image on the 8-bit scale. SSIM uses an $11\times11$ Gaussian window with standard deviation $1.5$, constants
\begin{equation}
    C_1=(0.01\cdot255)^2,
    \qquad
    C_2=(0.03\cdot255)^2,
\end{equation}
and the spatial average of the SSIM map, with equivalent MATLAB and Python convolution implementations.

Table \ref{tab:final-comparison} reports this fixed-realization comparison. K-SVD and LPG-PCA are included as SVD- or PCA-related reference points, but the controlled comparison for this paper is between the two SVD columns, where every pipeline component is shared except the rank-selection rule. In this setting, replacing residual-energy rank selection by SURE changes the average from $29.71/0.8302$ to $29.81/0.8358$, with improvements over energy matching in $24$ of $36$ PSNR entries and $23$ of $36$ SSIM entries. Because these are single noisy realizations and the average gain is small, these numbers should be read as modest descriptive evidence for the rank-selection replacement, not as a claim of statistical significance or superiority over existing denoising methods. The gains are strongest at the larger tested noise levels and are not uniform at $\sigma=10$. Figure \ref{fig:visual-house-sigma50} gives one qualitative high-noise example on House at $\sigma=50$.

\begin{figure}[H]
\centering
\includegraphics[width=0.65\textwidth]{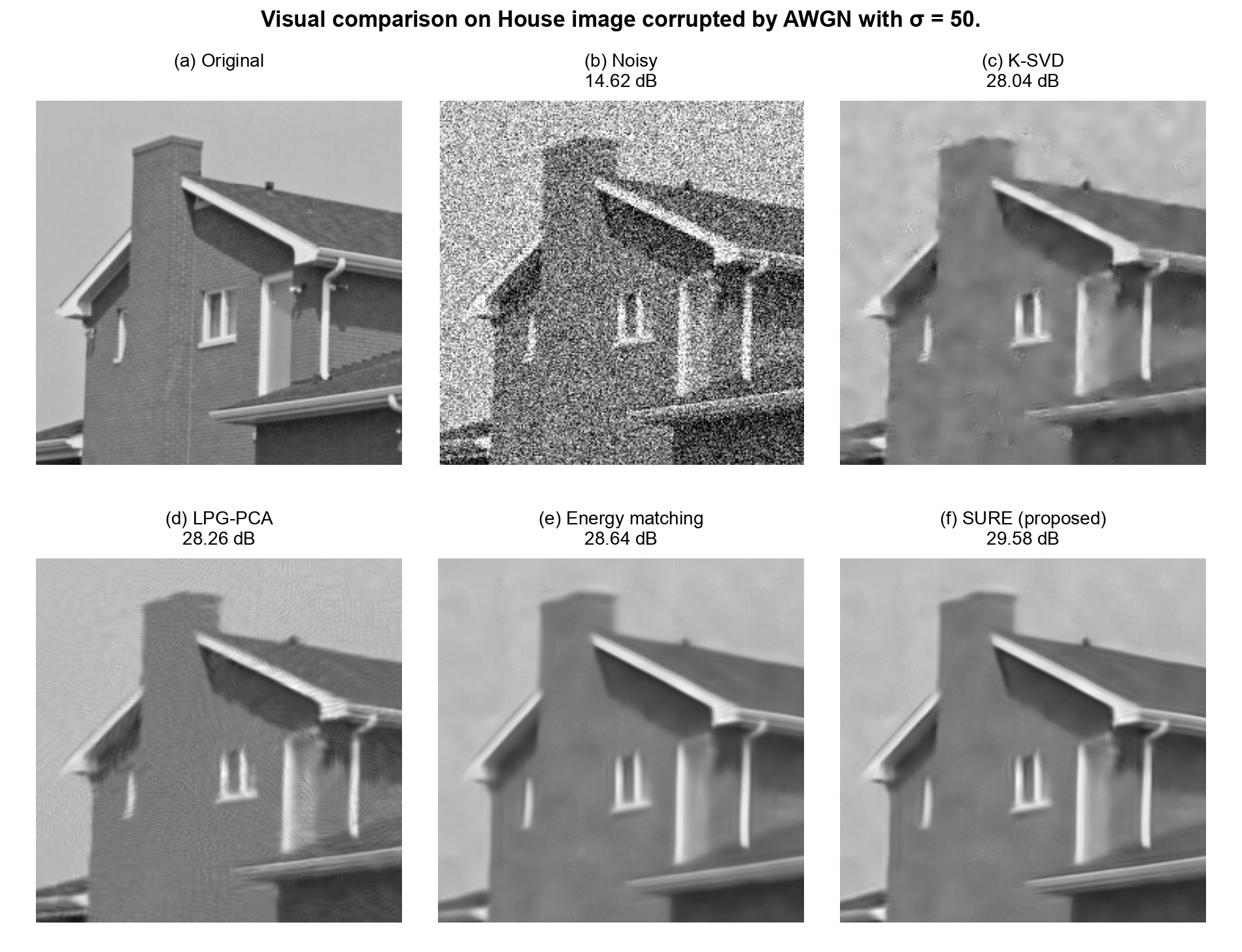}
\caption{Illustrative visual comparison on the House image corrupted by additive white Gaussian noise with $\sigma=50$. Panels show the original image, the shared noisy input, and the denoised outputs from K-SVD, LPG-PCA, energy matching, and the SURE-modified SVD pipeline. PSNR values are reported below the corresponding method labels for this single realization.}
\label{fig:visual-house-sigma50}
\end{figure}

\clearpage
\input{figures/table_final_comparison.tex}
\clearpage

\section{Conclusion and Discussion}

\paragraph{Contribution.}
This paper develops a smooth hard-threshold spectral estimator whose SURE is unbiased for every deterministic threshold under Gaussian noise. For a fixed observed matrix and finitely many separated candidates, its large-$\omega$ ordering yields a computable rank-selection score. The resulting procedure is a tuning rule subject to the fixed-threshold and post-selection distinctions established in Section 4.

\paragraph{Limitations.}
The limitations are as follows. The ordering argument is finite-candidate and pointwise: thresholds must be separated from the observed singular values, and the result does not justify continuous optimization over all $\lambda\in\R$. The limiting score ranks candidates but is not an unbiased risk estimate for discontinuous hard thresholding or for a data-selected threshold. In the image pipeline, block matching and overlapping aggregation go beyond the fixed-matrix Gaussian model used in the proof. The finite-$\omega$ and illustrative Set12 comparisons use fixed matched noisy realizations, the BSD68 Wilcoxon analysis uses one recorded realization for each image--noise pair, and the noise-level sensitivity study uses only three realizations per image. These experiments do not fully quantify variability over alternative independent noise seeds, so a broader multiple-seed image benchmark is left for future work. The runtime measurements are implementation-dependent and are used only to indicate relative overhead.

\paragraph{Future work.}
The present paper studies the direct matrix-denoising problem $Y=X_0+W$, corresponding to the special case in which the observation operator is the identity. A natural extension is to general linear inverse problems of the form
\begin{equation}
    y=\mathcal A(X_0)+\varepsilon,
\end{equation}
where $\mathcal A$ may represent missing-data sampling, blurring, downsampling, or partial Fourier sampling. In this setting, the smooth singular-value shrinkage developed here could serve as a spectral regularization step combined with a data-consistency step determined by $\mathcal A$ and its adjoint, thereby yielding an iterative reconstruction algorithm.

For example, in an accelerated magnetic resonance imaging (MRI) problem, one may consider
\begin{equation}
    \mathcal A_{\Omega}(X)
    =P_{\Omega}\mathcal F_{xy}(S\odot X),
\end{equation}
Here, $S$ denotes the coil-sensitivity maps, $\mathcal F_{xy}$ is the two-dimensional spatial Fourier transform, and $P_{\Omega}$ is the undersampling mask. This multi-coil compressed-sensing MRI formulation is described by Lustig et al. \cite{lustig2007sparse}. Developing this extension requires further analysis of smooth spectral updates at arbitrary proximal scales, stability of the undersampled operator on low-rank error sets, and parameter-selection methods appropriate for inverse problems, such as generalized SURE or validation using independent measurements. These theoretical questions, together with systematic experiments under different sampling strategies, are left for future work.

\section*{Acknowledgments and AI Disclosure}
Artificial-intelligence tools were used for auxiliary debugging assistance during preparation of the supplementary code, for limited assistance with manuscript language and formatting, and for preliminary literature-search support. The authors independently verified all mathematical claims, computational results, cited literature, and bibliographic information, reviewed and checked the supplementary code, are not aware of any issues affecting the reported results, and assume full responsibility for the correctness, reproducibility, licensing, and content of the supplementary code submission and for all content of this manuscript.

\end{document}

%% file: figures/algorithm_box.tex
\begin{algorithm}[H]
\caption{Proposed SURE-modified SVD denoising}
\label{alg:sure-svd}
{\scriptsize
\begin{algorithmic}[1]
\REQUIRE Noisy image $\mathbf y\in\R^{H\times W}$ on the 8-bit intensity scale, noise standard deviation $\tau$ on the same scale
\ENSURE Denoised image $\widehat{\mathbf x}$
\STATE \textbf{Stage 1: initial denoising.}
\STATE Extract overlapping $\sqrt m\times\sqrt m$ patches from $\mathbf y$; the reference-patch stride selects reference locations, and the search-window half-size restricts candidates. For each reference patch, select the $L$ closest patches by $S(\mathbf y_j,\mathbf y_c)=\|\mathbf y_j-\mathbf y_c\|_2^2$ and form $P_j=[\mathbf y_j,\mathbf y_{c,1},\ldots,\mathbf y_{c,L}]\in\R^{m\times n_j}$, where $n_j=L+1$.
\FOR{each group matrix $P_j$, $j=1,\ldots,C$}
    \STATE Compute $P_j=U_j\diagop(\sigma_{j,1},\ldots,\sigma_{j,k_j})V_j^T$, where $k_j=\min(m,n_j)$.
    \STATE Choose the SURE rank $\widehat h_j\in\arg\min_{h\in\{0,1,\ldots,k_j\}}S_h(P_j)$, where
    \[
    \begin{aligned}
    S_0(P_j)&=\|P_j\|_F^2-m n_j\tau^2,\\[-0.2em]
    S_h(P_j)&=-m n_j\tau^2+\sum_{i=h+1}^{k_j}\sigma_{j,i}^2 \\
    &\quad +2\tau^2\left\{|m-n_j|h+h
    +2\sum_{i=1}^{h}\sum_{\substack{\ell=1\\ \ell\neq i}}^{k_j}
    \frac{\sigma_{j,i}^2}{\sigma_{j,i}^2-\sigma_{j,\ell}^2}\right\},\quad h\geq1.
    \end{aligned}
    \]
    \STATE Set $\widehat Q_j=U_j\diagop(\sigma_{j,1},\ldots,\sigma_{j,\widehat h_j},0,\ldots,0)V_j^T$.
    \STATE Set $w_j=1-\widehat h_j/(L+1)$ if $\widehat h_j<L+1$, and $w_j=1/(L+1)$ otherwise.
\ENDFOR
\STATE Aggregate all $\widehat Q_j$ by weighted averaging to obtain $\widehat{\mathbf x}_0$.
\STATE \textbf{Back projection:} form $\widetilde{\mathbf y}=\widehat{\mathbf x}_0+\delta(\mathbf y-\widehat{\mathbf x}_0)$ with $\delta=0.5$, and update $\widetilde\tau=\gamma\sqrt{\max\{\tau^2-\|\widetilde{\mathbf y}-\widehat{\mathbf x}_0\|_F^2/(HW),0\}}$, where $\gamma$ is a scaling factor.
\STATE \textbf{Stage 2: refinement.} Repeat Stage~1 on $\widetilde{\mathbf y}$ with noise level $\widetilde\tau$ to obtain $\widehat{\mathbf x}$.
\end{algorithmic}
\medskip
\noindent\textbf{Parameters:} patch size $9\times9$ if $\tau<20$, $10\times10$ if $20\leq\tau<40$, and $11\times11$ if $\tau\geq40$; these thresholds use the 8-bit scale; similar patches $L=85$; search-window half-size $35$; reference-patch stride $3$.

\vspace{0.2em}
{\noindent\textbf{Remark:} The only modification relative to Guo et al.\ \cite{guo2016efficient} is the rank-selection rule: energy matching is replaced by SURE-based limiting-score minimization. Patch grouping, SVD factorisation, truncation, aggregation and back projection remain identical. The branch $w_j=1/(L+1)$ for a full-rank group is the positive fallback aggregation weight inherited from the original LRA-SVD rule, rather than the continuous extension $1-\widehat h_j/(L+1)=0$.\par}
}
\end{algorithm}

%% file: figures/table_bsd68_wilcoxon.tex
\begin{table}[t]
\centering
\caption{Paired BSD68 comparison between energy matching and the SURE-modified SVD pipeline using one recorded noisy realization per image--noise pair. Here $\Delta$ denotes SURE minus energy matching, and the $p$-values are one-sided Wilcoxon signed-rank tests for a positive signed-rank location shift at each noise level. Runtime entries are mean wall-clock seconds per image in this implementation.}
\label{tab:bsd68-wilcoxon}
{\scriptsize
\setlength{\tabcolsep}{3.0pt}
\renewcommand{\arraystretch}{0.92}
\resizebox{\textwidth}{!}{%
\begin{tabular}{@{}crrrrrrrr@{}}
\toprule
$\sigma$ & $N$ & $\Delta$PSNR & PSNR win & $p_{\rm PSNR}$ & $\Delta$SSIM & SSIM win & $p_{\rm SSIM}$ & Time EM/SURE \\
\midrule
10 & 68 & -0.517 & 4\% & 1.000 & -0.0068 & 10\% & 1.000 & 77.8/78.5 \\
30 & 68 & -0.050 & 44\% & 0.978 & -0.0016 & 44\% & 0.992 & 83.8/84.8 \\
50 & 68 & +0.118 & 65\% & 0.041 & +0.0129 & 68\% & 0.042 & 89.3/90.6 \\
\bottomrule
\end{tabular}%
}
}
\end{table}

%% file: figures/table_final_comparison.tex
\begin{table}[p]
\centering
\caption{Illustrative fixed-realization PSNR and SSIM comparison on Set12 under additive white Gaussian noise. Noise levels are reported on the 8-bit intensity scale. The controlled comparison of interest is between the energy-matching and SURE-modified SVD columns; K-SVD and LPG-PCA are included as SVD- or PCA-related reference points. Boldface marks the best PSNR and the best SSIM separately among the displayed methods in each image--noise row, with ties included.}
\label{tab:final-comparison}
{\tiny
\setlength{\tabcolsep}{2.0pt}
\renewcommand{\arraystretch}{0.86}
\resizebox{\textwidth}{!}{%
\begin{tabular}{@{}ll*{4}{cc}@{}}
\toprule
Image & $\sigma$ & \multicolumn{2}{c}{K-SVD} & \multicolumn{2}{c}{LPG-PCA} & \multicolumn{2}{c}{Energy matching} & \multicolumn{2}{c}{SURE proposed} \\
\cmidrule(lr){3-4}\cmidrule(lr){5-6}\cmidrule(lr){7-8}\cmidrule(l){9-10}
 & & PSNR & SSIM & PSNR & SSIM & PSNR & SSIM & PSNR & SSIM \\
\midrule
Cameraman & 10 & 33.64 & 0.9292 & 33.66 & 0.9258 & \textbf{34.08} & \textbf{0.9357} & 33.65 & 0.9345 \\
 & 30 & 28.01 & 0.8242 & 27.94 & 0.7958 & \textbf{28.26} & \textbf{0.8275} & \textbf{28.26} & 0.8237 \\
 & 50 & 25.53 & 0.7543 & 25.47 & 0.7009 & 25.32 & 0.7017 & \textbf{25.76} & \textbf{0.7684} \\
\addlinespace[0.15ex]
House & 10 & 36.01 & 0.9116 & 36.28 & 0.9162 & 36.57 & 0.9178 & \textbf{36.61} & \textbf{0.9190} \\
 & 30 & 31.26 & 0.8403 & 31.17 & 0.8106 & 31.81 & 0.8534 & \textbf{32.26} & \textbf{0.8584} \\
 & 50 & 28.05 & 0.7780 & 28.27 & 0.7151 & 28.65 & 0.8058 & \textbf{29.59} & \textbf{0.8203} \\
\addlinespace[0.15ex]
Pepper & 10 & 34.19 & 0.9278 & 34.06 & 0.9219 & 34.56 & 0.9314 & \textbf{34.64} & \textbf{0.9320} \\
 & 30 & 28.67 & 0.8443 & 28.40 & 0.8139 & 28.67 & 0.8481 & \textbf{28.94} & \textbf{0.8529} \\
 & 50 & 26.15 & 0.7851 & 25.72 & 0.7191 & 25.79 & 0.7881 & \textbf{26.42} & \textbf{0.7992} \\
\addlinespace[0.15ex]
Fishstar & 10 & 33.00 & 0.9306 & 33.17 & 0.9284 & \textbf{33.57} & \textbf{0.9340} & 33.11 & 0.9295 \\
 & 30 & 27.23 & 0.8153 & 27.31 & 0.8034 & 27.53 & 0.8224 & \textbf{27.68} & \textbf{0.8274} \\
 & 50 & 24.28 & 0.7195 & 24.41 & 0.6954 & 24.32 & 0.7225 & \textbf{24.53} & \textbf{0.7303} \\
\addlinespace[0.15ex]
Monarch & 10 & 33.75 & 0.9527 & 33.99 & 0.9526 & \textbf{34.59} & \textbf{0.9612} & 34.31 & 0.9598 \\
 & 30 & 27.80 & 0.8759 & 27.76 & 0.8582 & 28.08 & 0.8828 & \textbf{28.35} & \textbf{0.8929} \\
 & 50 & 25.17 & 0.8051 & 24.96 & 0.7610 & 25.14 & 0.8090 & \textbf{25.58} & \textbf{0.8225} \\
\addlinespace[0.15ex]
Airplane & 10 & 33.00 & 0.9249 & 33.00 & 0.9202 & \textbf{33.26} & \textbf{0.9283} & 32.65 & 0.9255 \\
 & 30 & 27.19 & 0.8336 & 27.10 & 0.7997 & \textbf{27.38} & \textbf{0.8391} & 27.32 & 0.8378 \\
 & 50 & 24.56 & 0.7606 & 24.55 & 0.6900 & 24.62 & 0.7674 & \textbf{24.73} & \textbf{0.7682} \\
\addlinespace[0.15ex]
Parrot & 10 & 33.18 & 0.9281 & 33.42 & 0.9261 & \textbf{33.65} & \textbf{0.9336} & 33.10 & 0.9292 \\
 & 30 & 27.54 & 0.8259 & 27.68 & 0.8020 & 27.97 & 0.8316 & \textbf{27.99} & \textbf{0.8324} \\
 & 50 & 25.40 & 0.7704 & 25.29 & 0.7095 & 25.77 & 0.7783 & \textbf{25.82} & \textbf{0.7810} \\
\addlinespace[0.15ex]
Barbara & 10 & 34.39 & 0.9361 & 35.02 & 0.9404 & \textbf{35.41} & \textbf{0.9446} & 35.29 & 0.9444 \\
 & 30 & 28.56 & 0.8266 & 29.24 & 0.8411 & 29.62 & 0.8619 & \textbf{29.98} & \textbf{0.8729} \\
 & 50 & 25.38 & 0.7170 & 26.41 & 0.7373 & 26.36 & 0.7654 & \textbf{26.94} & \textbf{0.7827} \\
\addlinespace[0.15ex]
Ship & 10 & 33.64 & 0.8868 & 33.72 & 0.8880 & \textbf{33.93} & \textbf{0.8896} & 33.74 & 0.8857 \\
 & 30 & 28.34 & 0.7472 & 28.27 & 0.7393 & 28.32 & 0.7474 & \textbf{28.46} & \textbf{0.7524} \\
 & 50 & 25.88 & 0.6656 & 25.87 & 0.6340 & 25.71 & 0.6643 & \textbf{26.05} & \textbf{0.6744} \\
\addlinespace[0.15ex]
Man & 10 & 33.54 & 0.9018 & 33.71 & 0.9025 & \textbf{33.96} & \textbf{0.9077} & 33.71 & 0.9058 \\
 & 30 & 28.17 & \textbf{0.7525} & 28.18 & 0.7413 & 28.16 & 0.7506 & \textbf{28.19} & 0.7506 \\
 & 50 & 26.02 & 0.6697 & 26.05 & 0.6356 & 25.87 & 0.6695 & \textbf{26.21} & \textbf{0.6809} \\
\addlinespace[0.15ex]
Couple & 10 & 33.52 & 0.9004 & 33.62 & 0.9011 & \textbf{33.93} & \textbf{0.9063} & 33.71 & 0.9013 \\
 & 30 & 27.94 & 0.7509 & 28.08 & 0.7542 & 28.18 & 0.7643 & \textbf{28.39} & \textbf{0.7726} \\
 & 50 & 25.31 & 0.6425 & 25.50 & 0.6343 & 25.25 & 0.6504 & \textbf{25.58} & \textbf{0.6665} \\
\midrule
Average &  & 29.39 & 0.8235 & 29.49 & 0.8042 & 29.71 & 0.8302 & \textbf{29.81} & \textbf{0.8358} \\
\bottomrule
\end{tabular}%
}
}
\end{table}